%% file: article.tex
\documentclass[11pt,a4paper]{article}

\input{pre}

\newcommand{\X}{\mathcal{X}}

\newcommand{\Oh}{\mathcal{O}}

\newcommand{\Het}{\mathrm{H}_{\text{ét}}}
\newcommand{\Ho}{\mathrm{H}}
\newcommand{\Gal}{\mathrm{Gal}}
\newcommand{\opp}{^{\mathrm{op}}}
\usepackage[OT2,T1]{fontenc}
\DeclareSymbolFont{cyrletters}{OT2}{wncyr}{m}{n}
\DeclareMathSymbol{\Sha}{\mathalpha}{cyrletters}{"58}

\title{Description of the strong approximation locus using Brauer-Manin obstruction for homogeneous spaces with commutative stabilizers}
\author{Victor de Vries and Haowen Zhang}

\date{}
\begin{document}
\maketitle
\begin{abstract}
    For a homogeneous space $X$ over a number field $k$, the Brauer-Manin obstruction has been used to study strong approximation for $X$ away from a finite set $S$ of places, and known results state that $X(k)$ is dense in the omitting-$S$ projection of the Brauer-Manin set $\pr_S(X(\A_k)^{\br})$, under certain assumptions. In order to completely understand the closure of $X(k)$ in the set of $S$-adelic points $X(\A_k^S)$, we ask: (i) whether $\pr_S(X(\A_k)^{\br})$ is closed in $X(\A_k^S)$; (ii) whether $X(k)$ is dense in the closed subset of $X(\A_k^S)$ cut out by elements in $\br X$ which induce zero evaluation maps at all the places in $S$. We also ask these questions considering only the algebraic Brauer group. We give answers to such questions for homogeneous spaces $X$ under semisimple simply connected groups with commutative stabilizers. 
\end{abstract}

\vspace{\baselineskip}
\section{Introduction}
Let $X$ be a variety over a number field $k$, and we suppose the set $X(k)$ of rational points is nonempty. We say that $X$ satisfies \textit{strong approximation away from $S$} for a finite set $S$ of places, if the diagonal image of $X(k)$ is dense in the $S$-adelic points $X(\A_k^S)$ (where the places in $S$ are omitted), meaning that we can simultaneously approximate finitely many $k_v$-points with $v\notin S$ by a single $k$-rational point with the
condition that this point is integral at all other finite places outside $S$. A classical case is the affine line $\BA_\QQ^1$ satisfying strong approximation away from $S=\{\infty\}$ containing only the archimedean place, resulting from the Chinese Remainder Theorem. Moreover, a semisimple, almost simple, simply connected linear algebraic group $G$ such that $\prod_{v\in S}G(k_v)$ is not compact satisfies strong approximation away from $S$ (see \cite[Theorem~7.21]{platonov1993algebraic}). 
\par
Colliot-Thélène and Xu \cite{colliot2009brauer} used the Brauer-Manin obstruction to study strong approximation problems for homogeneous spaces $X=G/H$ under semisimple simply connected groups $G$ with connected or finite commutative stabilizers $H$. They proved that if $G$ satisfies strong approximation away from $S$, then $X(k)$ is dense in $\pr_S(X(\A_k)^{\br})$, the image of the Brauer-Manin set $X(\A_k)^{\br}$ under the projection $\pr_S\colon X(\A_k)\rightarrow X(\A_k^S)$. 
Borovoi and Demarche \cite{borovoi2013manin} then generalized this result to homogeneous spaces $X=G/H$ under connected $k$-groups $G$ with connected stabilizers $H$, and proved that 
$X(k)$ is dense in $\pr_S(X(\A_k)^{\br})$ for an $S$ which contains all archimedean places, assuming that the semisimple simply connected part $\Gsc$ satisfies strong approximation away from~$S$, and that a certain Tate-Shafarevich group is finite.
\par
In order to give a precise description of the set of adelic points away from $S$ that can be approximated by rational points, that is, the closure $\overline{X(k)}^S$ of $X(k)$ inside $X(\A_k^S)$, we naturally ask if~$\pr_S(X(\A_k)^{\br})$ is closed in $X(\A_k^S)$. While $\pr_S(X(\A_k)^{\br})$ is closed when $\br X/\br k$ is finite,
the closedness of $\pr_S(X(\A_k)^{\br})$ can fail for a general $X$ (for example~$X=\gm$ with $S=\{2,\infty\}$, see \cite[Proposition~4.8]{demeio2022etale}).
We show that in the case of a homogeneous space $X$ under a semisimple simply connected group $G$, the set $\pr_S(X(\A_k)^{\br})$ is indeed closed. The group $\br X/\br k$ can be infinite when the stabilizers of the homogeneous space $X$ are not connected.

\begin{thm}[Theorem \ref{closed} and Theorem \ref{projectioncommutativecase}]
Let $X$ be a homogeneous space of a semisimple and simply connected algebraic group $G$ over a number field $k$. Then the set~$\pr_S(X(\A_k)^{\br})\subseteq X(\A_k^S)$ is closed. If we suppose moreover that $X=G/H$ with commutative $H$, and $G$ satisfies strong approximation away from $S$, then $\overline{X(k)}^S=\pr_S(X(\A_k)^{\br})$.
\end{thm}
For a general variety $X$, there is also another point of view to remedy the potential non-closedness of $\pr_S(X(\A_k)^{\br})$: 
one can define an \textit{a priori} closed subset in $X(\A_k^S)$ containing $X(k)$ using the Brauer-Manin pairing, and hope to relate it to $\overline{X(k)}^S$. The following subgroup of ``trivial at $S$'' elements 
\[\br_S X\coloneq \ker(\br X\rightarrow\prod_{v\in S}\br X_v)\] has been considered, as in  \cite{harari2010brauer}\cite{demeio2022etale} etc, and $\br_S X$ cuts out a closed subset $X(\A_k ^S)^{\br_S}$ in $X(\A_k ^S)$. In \cite{harari2008defaut}, Harari considered algebraic tori $X$ over $k$, and showed $X(k)$ is dense in $\pr_S(X(\A)^{\br})$ for~$S$ containing all archimedean places; in a later paper together with Voloch \cite{harari2010brauer}, they established more precisely that $\overline{X(k)}^S=X(\A_k^S)^{\br_S}$. In the case of homogeneous spaces $X=G/H$ with $G$ and $H$ connected, Demeio \cite{demeio2022etale} showed that $\overline{X(k)}^S=X(\A_k^S)_\bullet^{\br_S}$ for~$S$ \textit{containing no archimedean places}, with the same assumptions on $\Gsc$ and the Tate-Shafarevich group as in \cite{borovoi2013manin}, and $X(\A_k^ S)_\bullet$ denotes the modified adelic space where each $X(k_v)$ for $v$ archimedean is collapsed to its connected components; for arbitrary stabilizers, Demeio gave results using the étale-Brauer-Manin obstruction, and there are known examples in \cite{demarche2017obstructions} with finite non-commutative stabilizers showing that the Brauer-Manin obstruction is not enough.
\par

It is natural to wonder if we can generalize the equality $\overline{X(k)}^S=X(\A_k^S)^ {\br_S}$ to other situations with stabilizers that are not necessarily connected, for arbitrary $S$ possibly containing archimedean places too (noting that it is common in the literature to take $S$ containing all archimedean places in strong approximation problems, contrary to the assumption in \cite{demeio2022etale}). We give examples showing that this is not true in general, when $S$ contains real places. Hence we may need finer information of the Brauer-Manin pairing, and we propose a smaller closed subset $X(\A_k^ S)^ {\br_S^ \prime}$ containing $X(k)$  cut out by
\[\br_S^ \prime X\coloneq \{A\in\br X|\text{ the induced evaluation map }\ev_A\colon X(k_v)\rightarrow\br k_v\text{ is zero for all }v\in S\}.\]
We replace
$\br$ by $\br_1$ when we consider only elements in $\br_1 X\coloneq \ker(\br X\rightarrow\br \overline X)$, the algebraic Brauer group. We have the following chains of inclusions:
\begin{equation*}
 \begin{split}\overline{X(k)}^S\subseteq X(\A_k^S)^ {\br^ \prime_S}&\subseteq X(\A_k^S)^ {\br_S} \subseteq X(\A_k^S)^ {\br_{1,S}} \\
 \overline{X(k)}^S\subseteq X(\A_k^S)^ {\br^ \prime_S}&\subseteq X(\A_k^S)^ {\br^ \prime_{1,S}}.\end{split}\end{equation*}
\par

For homogeneous spaces of semisimple simply connected groups with commutative stabilizers, we give precise conditions under which all the inclusions in each chain become equalities, and under these conditions we characterize the adelic points away from $S$ which can be approximated by rational points.

\begin{thm}[Theorem \ref{main3}] Let $X=G/H$ be a homogeneous space for a semisimple and simply connected algebraic group $G$ over a number field $k$ with commutative stabilizer $H$. Let $S$ be a finite set of places such that $G$ satisfies strong approximation away from $S$. Then we have
\begin{enumerate}[label=(\alph*)]

\item $\overline{X(k)}^S= X(\A_k^S)^ {\br^ \prime_S}= X(\A_k^S)^ {\br_S} = X(\A_k^S)^ {\br_{1,S}}$ if and only if \begin{equation}\label{brauer1intro}\prod_{v\in S_\R}\Ho^1(k_v,H)=r(\Sha^1_S(k,H))+\prod_{v\in S_\R}\delta_v(X(k_v)).\tag{*}\end{equation}
\item $\overline{X(k)}^S= X(\A_k^S)^ {\br^ \prime_S}= X(\A_k^S)^ {\br^ \prime_{1,S}}$ if and only if \begin{equation}\label{brauer1'intro}\prod_{v\in S_\R}\langle\delta_v(X(k_v))\rangle\subseteq r(\Sha^1_S(k,H))+\prod_{v\in S_\R}\delta_v(X(k_v)).\tag{**}\end{equation}
\end{enumerate} Here $S_{\R}$ is the set of real places in $S$, the map $\delta_v\colon X(k_v)\to \Ho^1(k_v,H)$ is the connecting map, $\langle\delta_v(X(k_v))\rangle$ is the subgroup generated by $\delta_v(X(k_v))$ inside $\Ho^1(k_v,H)$, and $r$ denotes the restriction map $r\colon\Ho^1(k,H)\to \prod_{v\in S_\R}\Ho^1(k_v,H)$, and $\Sha^1_S(k,H)\coloneq \ker(\Ho^1(k,H)\to~\prod_{v\notin S}\Ho^1(k_v,H))$.

\end{thm}

In particular, the condition (\ref{brauer1'intro}) is satisfied when $\delta_v(X(k_v))$ is already a subgroup in $\Ho^1(k_v,H)$ for all $v\in S_{\R}$, and this is the case when $H$ is central in $G$, for example, when $X$ is an adjoint group. Conditions (\ref{brauer1intro}) and (\ref{brauer1'intro}) are automatically satisfied when $S$ does not contain real places, or when $\Ho^1(k_v,G)$ is trivial for all $v\in S_{\R}$ because $\delta_v$ then becomes surjective.
\par We also establish examples with toric stabilizers and examples with finite commutative stabilizers to show that when the conditions (\ref{brauer1intro}) and (\ref{brauer1'intro}) are not satisfied, how the equalities in the chains can fail at intermediate steps. More precisely, we give examples in
\begin{itemize}
\item Propositions \ref{egSO} and \ref{egSpin/T} showing $\overline{X(k)}^S= X(\A_k^S)^ {\br^ \prime_S}\subsetneq X(\A_k^S)^ {\br_S},$
\item Proposition \ref{egSU/T2} showing $\overline{X(k)}^S= X(\A_k^S)^ {\br^ \prime_S}\subsetneq X(\A_k^S)^ {\br^{\prime}_{1,S}},$
\item Proposition \ref{extorus} and \ref{egSU/T4} showing $\overline{X(k)}^S\subsetneq X(\A_k^S)^ {\br^ \prime_S}.$
\end{itemize}
As a consequence, we know that the Brauer-Manin set cut out by $\br_S X$ and $\br_S^{\prime} X$ can be different, and that transcendental elements in the Brauer group can play a role. 
\par At the end, we discuss possible generalizations to homogeneous spaces of connected linear groups with commutative stabilizers, presenting conditional results and clarifying the challenges in reaching the fully general case.

\section{Notation and preliminaries}
\label{notation}

Throughout the paper we will use the following notation.
\begin{itemize}
    \item $k$ is a number field with set of places $\Omega$ and set of real places $\Omega_\R$
    \item $k_v$ is the completion of $k$ at the place $v$, with ring of integers $\mathcal O_v$ if $v$ is non-archimedean
    \item $S_\R\coloneq S\cap\Omega_\R$ for a subset $S\subseteq\Omega$
    \item $\mathcal O_S$ is the ring of $S$-integers, that is, the ring of elements in $k$ which are integral outside $S$
    \item $X$ is a $k$-variety, and $\mathcal X$ is a separated $\mathcal O_T$-scheme of finite type such that the generic fiber of $\mathcal{X}$ is $X$ (such an $\mathcal{X}$ is called an $\Oh_T$-model of $X$)
    \item $X(\A_k^S)\coloneq \prod_{v\in\Omega\backslash S}^\prime(X(k_v),\mathcal X(\mathcal O_v))$ is the set of $S$-adelic points on $X$ equipped with the restricted product topology; we write $X(\A_k)$ for $S=\emptyset$ and $\pr_S\colon X(\A_k)\to X(\A_k^S)$ is the natural projection map
    \item $\overline{X(k)}^S$ is the closure of $X(k)$ in $X(\A_k^S)$    
    \item $\br X\coloneq \Ho^2_{\et}(X,\Gm)$ is the (cohomological) Brauer group of $X$
    \item $\br_1 X\coloneq \ker(\br X\to\br\overline X)$ is the algebraic Brauer group of $X$
    \item $\br_S X\coloneq\ker(\br X\rightarrow\prod_{v\in S}\br X_{k_v})$
    \item $\br_S^ \prime X\coloneq \{A\in\br X|\text{ the induced evaluation map }\ev_A\colon X(k_v)\rightarrow\br k_v\text{ is zero for all }v\in~S\}$
    \item $\br_{1,S}X\coloneq \br_1 X\cap \br_S X$
    \item $\br^\prime_{1,S}X\coloneq \br_1 X\cap \br^\prime_S X$
    \item $\Ho^1_S(k,-)\coloneq \ker(\Ho^1(k,-)\to {\prod}_{v\in S}\Ho^1(k_v,-))$
    \item $\Sha^1_S(k,-)\coloneq \ker(\Ho^1(k,-)\to {\prod}_{v\notin S}\Ho^1(k_v,-))$ and we write $\Sha^1(k,-)$ for $S=\emptyset$
    \item $r_v$ is the restriction map $\Ho^1(k,-)\to \Ho^1(k_v,-)$ on Galois cohomology 
\end{itemize}

Let $G$ be 
an algebraic group over $k$, and let $X$ be a homogeneous space for $G$ with point $o\in X(k)$ with stabilizer $H$. 
For any field $L/k$ one has an exact sequence of pointed sets that comes from non-abelian Galois cohomology:
\begin{equation*}
    G(L)\to X(L)\xrightarrow{\delta_L}\Ho^1(L,H(\overline{L}))\to \Ho^1(L,G(\overline{L})).
\end{equation*}
When $L=k_v$, we denote by $\delta_v$ the connecting map $X(k_v)\rightarrow\Ho^1(k_v,H)$.
In order to make a similar exact sequence on the adèles of $k$, we have to make sure that we can give a well-defined restricted product.  We do this as follows:

There exists a finite set of places $T$ of $k$ containing all archimedean ones satisfying the following conditions: There are integral smooth models $\mathcal{G}$, $\mathcal{X}$, $\mathcal{H}$ of $G$, $X$ and $H$ over $\Spec(\Oh_S)$ (here by models of $\mathcal{G}$ and $\mathcal{H}$ we mean models of group schemes) such that $\mathcal{G}\to \mathcal{X}$ extends $G\to X$ and has stabilizer $\mathcal{H}$ at $o\in \X(\Oh_S)$ (the point $o$ has been lifted to an $\Oh_S$-point). Using this we make the following definition.

\begin{df}
Let $T$ and $\mathcal{G}$ be as above and let $S$ be an other finite set of places. Define the set $\prod'_{v\notin S}\Ho^1(k_v,G)$ to be the set of points $(y_v)_{v\notin S}$ in $\prod_{v\notin S}\Ho^1(k_v,G)$ such that for all but finitely many $v\notin T\cup S$, there is a $z_v\in \Het^1(\Oh_v,\mathcal{G})$ mapping to $y_v$.
\end{df}

The definition is independent of the choices of $T$ and $\mathcal{G}$ since for a different model of $G$, there is a large enough $T'$ such that the new model and $\mathcal{G}$ are isomorphic over $\Oh_{T'}$. We make $\prod'_{v\notin S}\Ho^1(k_v,G)$ into a topological space by endowing each $\Ho^1(k_v,G)$ with the discrete topology and then giving $\prod'_{v\notin S}\Ho^1(k_v,G)$ the restricted product topology.

\begin{rmk}\label{2.2}
The reader may wonder if $\Het^1(\Oh_v,\mathcal{G})\to \Ho^1(k_v,G)$ is actually injective? This is indeed true for all $v\notin T$ with $T$ as above under the condition that $G$ is either connected or commutative. For the connected case $\Het^1(\Oh_v,\mathcal{G})$ is trivial, which follows from the fact that $\Het^1(\Oh_v,\mathcal{G})=\Ho^1(\mathbb{F}_v,\mathcal{G}_v)=\{*\}$ where $\mathcal{G}_v$ denotes the reduction. The first equality here is by Hensel's lemma and the second one is Lang's theorem. For the commutative case, by a diagram chase it suffices to prove injectivity in the case of a finite commutative group~$F$. Injectivity of $\Het^1(\Oh_v,\mathcal{F})\to \Ho^1(k_v,F)$ follows from the valuative criterion for properness. 
\end{rmk}

We now make the following proposition.

\begin{prop}\label{pointedsets}
For a finite set $S$ of places, there is an exact sequence of pointed topological spaces: \begin{equation*}
G(\A_k^S)\to X(\A_k^S)\to  {\prod}^\prime_{v\notin S}\Ho^1(k_v,H)\to  {\prod}^\prime_{v\notin S}\Ho^1(k_v,G).
\end{equation*}
Moreover, elements in $X(\A_k^S)$ have the same image in ${\prod}^\prime_{v\notin S}\Ho^1(k_v,H)$ if and only if they are in the same orbit under the action of $G(\A_k^S)$.
\begin{proof}
There is an exact sequence $G(k_v)\to X(k_v)\to \Ho^1(k_v,H)\to \Ho^1(k_v,G)$, and there exists a large enough $T\subseteq\Omega$ such that for all $v\notin T$, one also has such an exact sequence for $\Oh_v$-points with respect to the relevant models (where with $\Ho^1(\Oh_v,-)$ we mean the étale cohomology) (see \cite[Proposition III.3.2.2]{cohomologienonabelienne}. This shows the exactness of the sequence of adelic points. The statement in the proposition concerning the orbits is a direct consequence of Corollary III.3.2.3 in \cite{cohomologienonabelienne}. By \cite[Proposition 4.2]{ČESNAVIČIUS_2015} all of the maps are continuous. 
\end{proof}
\end{prop}

The last term in the exact sequence is quite simple by the following theorem.

\begin{thm}\label{kneserharder}\cite[Theorem 6.4 and Theorem 6.6]{platonov1993algebraic}
Let $G$ be a semisimple and simply connected group over a non-archimedean local field $k_v$. The cardinality of $\Ho^1(k_v,G)$ is $1$. If $G$ is defined over a number field $k$, the natural map $\Ho^1(k,G)\to \prod_{v\in \Omega_\R}\Ho^1(k_v,G)$ is a bijection. 
\end{thm}

\section{Closedness of the projection $\pr_S(X(\A_k)^{\br})$}

There is a natural pairing 
\begin{align*}
X(\A_k)\times \br X&\rightarrow \QZ \\
(\{x_v\},A)&\mapsto \sum_{v\in\Omega}\inv_v(\ev_A (x_v))
\end{align*}
known as the Brauer-Manin pairing, 
where $\ev_A:\br X\rightarrow \br k_v$ is the evaluation map induced by $A$, and $\inv_v\colon \br k_v\hookrightarrow \QZ$ is the local invariant map from Class Field Theory. We denote by $X(\A_k
)^{\br}$ (resp. $X(\A_k
)^{\br_1}$) the set of adelic points which are orthogonal to $\br X$ (resp. $\br_1 X$). The set of rational points $X(k)$ is contained in $X(\A_k
)^{\br}\subseteq X(\A_k
)^{\br_1}$ by the Albert-Brauer-Hasse-Noether exact sequence
\[0\rightarrow\br k\rightarrow\bigoplus_{v\in \Omega}\br k_v\xrightarrow{\Sigma\inv_v}\QZ\rightarrow 0.\]
\par Colliot-Thélène and Xu \cite{colliot2009brauer} proved that if $X=G/H$ is a homogeneous space where $G$ is semisimple and simply connected, satisfying strong approximation away from a finite set of places $S$, and $H$ is connected or finite commutative, then the set of rational points $X(k)$ is dense in $\pr_S(X(\A_k)^{\br})$. Note that the case of finite commutative stabilizers can be generalized to stabilizers of multiplicative type (see \cite[Theorem~4.1]{harari2013descent}).
\par In this section, we study the question of the closedness of $\pr_S(X(\A_k)^{\br})$ in order to completely characterize $\overline{X(k)}^S$. For homogeneous $X=G/H$ under semisimple simply connected groups $G$
 with \textit{connected} stabilizers $H$, we have that $\br X/\br k$ is isomorphic to $\Pic H$ (see \cite[Proposition~2.10.(ii)]{colliot2009brauer}) which is finite. A calculation shows that finiteness of $\br X/\br k$ implies that $\pr_S(X(\A_k)^{\br})$ is indeed closed in $X(\A_k^S)$. The group $\br X/\br k$ can be infinite when the stabilizers are not connected. However, we show that in fact $\pr_S(X(\A_k)^{\br})$ is closed in $X(\A_k^S)$ for all homogeneous spaces of semisimple simply connected groups.

\begin{lem}\label{factorviaH1}
Let $X$ be a homogeneous space for a semisimple and simply connected group $G$ with stabilizer $H$ at $o\in X(F)$ where $F$ is a field of characteristic $0$. For $x\in X(F)$ and $\beta\in \br{X}$, $x^*\beta\in \br F$ depends only on the image of $x$ in $\Ho^1(F,H)$.
\begin{proof}
The statement is equivalent to $(g\cdot x)^*\beta=x^*\beta$ for all $g\in G(F)$, $x\in X(F)$ and $\beta\in \br{X}$. Consider the following composition, which equals $(g\cdot x)\colon\Spec F\to X$: $$\Spec F\overset{(g,x)}{\to} G\times_F \{x\}\overset{\iota_x}{\to} G\times_F X\overset{\rho}{\to} X$$
This implies that $(g\cdot x)^*\beta=(g,x)^*(\iota_x^*(\rho^*\beta))$ for all $g\in G(F)$. But this term does not depend on $g$, since $G\times_F \{x\}\cong G$ and because the canonical map $\br F\to \br G$ is an isomorphism by $G$ being semisimple and simply connected (cf. \cite[Proposition~9.2.1]{colliot2021brauer}).
\end{proof}
\end{lem}

Using the previous lemma we can make the following definition.

\begin{df}
Let $X,H$ and $G$ be as in the previous lemma. Define the set $P$ by: $$P:=\mathrm{Im}\left( X(\A_k)\to {\prod}^\prime \Ho^1(k_v,H)\right)$$  We set $P^S:=\mathrm{Im}\left(P\to {\prod}^\prime_{v\notin S}\Ho^1(k_v,H)\right)$. We endow $P$ (resp. $P^S$) with the subspace topology from ${\prod}^\prime \Ho^1(k_v,H)$ (resp. ${\prod}^\prime_{v\notin S} \Ho^1(k_v,H)$).
\end{df}

\begin{rmk}
We can realize $P$ as a restricted product with factors $I_v=\mathrm{Im}\left(X(k_v)\to \Ho^1(k_v,H)\right)$ in an obvious way.
As a consequence of Theorem \ref{kneserharder} we have $I_v=\Ho^1(k_v,H)$ for all nonarchimedean places $v$. For real places $v$, the set $I_v$ corresponds naturally to the set of path components of $X(k_v)$ for the analytic topology (see \cite[Lemma 16.1]{borovoi2016realhomogenousspacesgalois}).
\end{rmk}

\begin{lem}\label{pairingwithP}
The Brauer-Manin pairing $X(\A_k)\times \br X\rightarrow \QZ$ factors through a well-defined pairing $P\times \br X\rightarrow \QZ $. For $A\in \br X$, the induced evaluation map $\ev_A: P\rightarrow \QZ$ is continuous.\end{lem}
\begin{proof}
The first statement is a direct consequence of 
Lemma \ref{factorviaH1}. Now we show the continuity. Let $T\subseteq \Omega$ be a big enough finite set of places such that $A$ comes from $\br \mathcal X_T$ for an $\Oh_T$-model~$\mathcal X$ of $X$ and $H$ extends to $\mathcal H$ over $\Oh_T$. Enlarging $T$ if necessary, we can suppose moreover that~$\mathcal X(\Oh_v)\rightarrow \Ho^1(\Oh_v,\mathcal H)$ is surjective: we have the exact sequence of pointed sets (cf. \cite[Proposition III.3.2.2]{cohomologienonabelienne}) $$\mathcal X(\Oh_v)\rightarrow \Ho^1(\Oh_v,\mathcal H)\rightarrow \Ho^1(\Oh_v,\mathcal G)$$ and $\Ho^1(\Oh_v,\mathcal G)$ vanishes for almost all $v\in \Omega$ (see Remark \ref{2.2}). The sets of the form $$P\bigcap\left(\prod_{v\in T^\prime}U_v\times\prod_{v\notin T^\prime}\Ho^1(\Oh_v,\mathcal H)\right)$$
for $T^\prime\supseteq T$ and subset $U_v\subseteq \Ho^1(k_v,H)$ form a basis of open sets of $P$. Hence it suffices to prove the continuity of $\ev_A$ on each of these open subsets. Since $\ev_A:\mathcal X(\mathcal O_v)\rightarrow \QZ$ is zero and $X(\mathcal O_v)\rightarrow \Ho^1(\Oh_v,\mathcal H)$ is surjective, the evaluation map $\ev_A: \Ho^1(\Oh_v,\mathcal H)\rightarrow\QZ$ is zero for all $v\notin T^\prime$. The evaluation map $\ev_A: \Ho^1(k_v,H)\rightarrow \QZ$ is continuous since $\Ho^1(k_v,H)$ is discrete. Therefore, the evaluation map $\ev_A: P\rightarrow \QZ$ is continuous.
\end{proof}
Using this we derive the following theorem.

\begin{thm}\label{closed}
Let $X$ be a homogeneous space of a semisimple and simply connected algebraic group $G$ over a number field $k$. Then the set~$\pr_S(X(\A_k)^{\br})\subseteq X(\A_k^S)$ is closed. 
\begin{proof}
First of all, note that there is a commutative diagram:
\begin{equation*}
\begin{tikzcd}
X(\A_k) \arrow[r, "\delta"] \arrow[d, "\pr_S"] & P \arrow[d, "\pi_S"] \\
X(\A_k^S) \arrow[r, "\delta_S"]                & P^S.                 
\end{tikzcd}
\end{equation*}
We claim that $\pr_S(X(\A_k)^{\br})=\delta_S^{-1}(\pi_S(P^{\br}))$. Indeed, if $(x_v)_{v\notin S}\in X(\A_k^S)$ can be completed to a point $(x_v)_v$ that is orthogonal to $\br{X}$, we have that $\delta_S((x_v)_{v\notin S})=\pi_S(\delta((x_v)_v))$. On the other hand, if $\delta_S((x_v)_{v\notin S})\in P^S$ can be completed to a point in $(y_v)_v\in P$ that is orthogonal to~$\br{X}$, simply lift $(y_v)_{v\in S}$ and use the lifts to complete $(x_v)_{v\notin S}$ to an adelic point. This proves the claim.
\par Since $X(\A_k^S)\to P^S$ is continuous, it will suffice to show that $\pi_S(P^{\br})$ is closed in $P^S$. By Lemma \ref{pairingwithP}, the set $P^{\br}$ is closed in $P$ (in fact, we can show that $P^{\br}$ is discrete and closed in $P$, see Corollary \ref{multdiscrete}). Since $\Ho^1(k_v,H)$ is finite (\cite[III.\S4. Theorem 4]{serre1994cohomologie}), the map $\pi_S$ is a covering projection with finite fibers, and is thus closed. Therefore, $\pi_S(P^{\br})$ is closed, and so is $\pr_S(X(\A_k)^{\br})$.

\end{proof}
\end{thm}


\section{Homogeneous spaces with commutative stabilizers}

Let $X=G/H$ with $G$ a semisimple simply connected group over $k$ and $H$ commutative. Since~$k$ has characteristic $0$, we have a unique decomposition $H=M
\times \Ga^r$ of $H$ into a largest unipotent subgroup isomorphic to $\Ga^r$ and a subgroup $M$ of multiplicative type (see \cite[Theorem 16.13 and Corollary 14.33]{milne2017algebraic}). Let $\hat{M}:=\Hom(M_{\overline{k}},\G_{m,\overline{k}})$ be the (geometric) character group of $M$. This is a finitely generated discrete Galois module, which we can view as a commutative $k$-group scheme locally of finite type.
For each place $v$, we have the cup-product pairing 
\[H^1(k_v,M)\times H^1(k_v,\hat M)\rightarrow\br k_v\addtag\label{PTpairing}\]
which is a perfect pairing of finite groups by the Poitou-Tate duality (see \cite[Corollary 2.3 and Theorem 2.13 in Chapter I ]{milne2006arithmetic}). This induces a pairing 
\[\left({\prod}_{v\in\Omega}^{\prime}H^1(k_v,M)\right)\times H^1(k,\Hat M)\rightarrow\QZ,\addtag\label{PairingH1}\] 
and we have a commutative diagram of continuous homomorphisms with exact rows
\[
\begin{tikzcd}
{\Ho^1(k,M)} \arrow[r] \arrow[d, equal] & {{\prod}_{v\notin S}^\prime \Ho^1(k_v,M)} \arrow[r]                       &   {H^1_S(k,\hat M)^*}                                                                 \\
{\Ho^1(k,M)} \arrow[r]                                & {{\prod}_{v\in\Omega}^\prime \Ho^1(k_v,M)} \arrow[u, two heads] \arrow[r] & {\Ho^1(k,\hat M)^*}  \arrow[u, two heads] \\                                                      & {\prod_{v\in S}\Ho^1(k_v,M)} \arrow[u, hook] \arrow[r, "\simeq"]          & {(\prod_{v\in S}\Ho^1(k_v,\hat M))^*} \arrow[u, hook]                                   
\end{tikzcd}\addtag\label{Poitou-Tate}
\]
where $A^*\coloneq \Hom(A,\QZ)$ denotes the Pontryagin dual of $A$, and $\Ho_S^1(k,-)$ is
as defined in \S\ref{notation}. Indeed, the middle row is the Poitou-Tate exact sequence (see \cite[Théorème~6.3]{demarche2011suites}). The first row is induced by the second row and is exact by a diagram chasing argument.
\par We have $\Ho^1(F,H)\simeq \Ho^1(F,M)\times \Ho^1(F,\Ga^r)= \Ho^1(F,M)$ for any field $F$ because of the vanishing of $H^1(F,\Ga)$ (see \cite[Proposition 1, Chapter II]{serre1994cohomologie}). Hence we get a connecting map $X(F)\to \Ho^1(F,M)$ through the composition $X(F)\to\Ho^1(F,H)=\Ho^1(F,M)$. Using this map for $F=k_v$, we have that the pairing (\ref{PTpairing}) is compatible with the Brauer-Manin pairing, in the sense that the following diagram 
\begin{equation}\label{br1pairing}
\begin{tikzcd}
X(k_v) \arrow[d] \arrow[r,phantom,"\times" description] & \br X \arrow[r] \arrow[d,leftarrow]      & \br 
 k_v \arrow[d,Rightarrow, no head] \\
\Ho^1(k_v,M)                                \arrow[r,phantom,"\times" description] & {\Ho^1(k,\hat M)} \arrow[r]  & \br k_v     
\end{tikzcd}
\end{equation}
is commutative. There is a natural isomorphism $\Ho^1(k,\hat M)\simeq \ker(o^*\colon\br_1 X\to \br k)$ and hence $\Ho^1(k,\hat M)$ may be identified with $\br_1 X/\br k$ (see \cite[Propositions~2.7 and~2.12]{colliot2009brauer}).
\begin{thm}\label{projectioncommutativecase}
Let $X$ be a homogeneous space as defined above, that is, $X=G/H$ with $G$ a semisimple simply connected group over $k$ and $H$ commutative.  If $G$ satisfies strong approximation away from $S$, then $\overline{X(k)}^S=\pr_S(X(\A_k)^{\br})=\pr_S(X(\A_k)^{\br_1})$.
\end{thm}
\begin{proof}
We first prove that $\pr_S(X(\A_k)^{\br})\subseteq \overline{X(k)}^S$. For this we use the same argument as in \cite[\S4]{colliot2009brauer} where they proved the case for finite commutative stabilizers $H$. 
The diagram (\ref{br1pairing}) and the exact sequences in (\ref{Poitou-Tate}) and Proposition \ref{pointedsets} give the commutative diagram: 
\[\begin{tikzcd}
    & G(k) \arrow[d] \arrow[r] & G(\A_k) \arrow[d]\\
    & X(k) \arrow[d] \arrow[r] & X(\A_k) \arrow[d, "\delta"] \arrow[r] & (\br_1 X/\br k)^*\arrow[d]\\
    & \Ho^1(k,M)\arrow[d] \arrow[r] & \prod'_{v\in \Omega_k} \Ho^1(k_v, M) \arrow[d] \arrow[r] & (H^1(k,\hat M))^*\\
    & \Ho^1(k,G) \arrow[r] & {\prod'_{v\in \Omega_k}}\Ho^1(k_v, G).
\end{tikzcd}\]
This shows that the image of a point $(x_v)_v\in X(\A_k)^{\br_1}$ under $\delta$ lifts to a point $y\in \Ho^1(k,M)$, such that $y$ vanishes in $\Ho^1(k,G)$. Therefore $y$ lifts to $x\in X(k)$ which differs from $(x_v)_v$ by an element of $G(\A_k)$, that is, $(x_v)_v=x\cdot (g_v)_v$ for some $(g_v)_v\in G(\A_k)$. The $k$-point $x$ induces a morphism $\varphi_x:G\to X$ defined by $g\mapsto g\cdot x$. Now take any open $V\subset X(\A_k^S)$ which contains $(x_v)_{v\notin S}$. Since $X(k)\subset \pr_S(X(\A_k)^{\br_1})$, it suffices to show that $X(k)\cap V$ is nonempty in order to conclude that $X(k)$ is dense in $\pr_S(X(\A_k)^{\br_1})$. We consider $\varphi^{-1}(V)\subset G(\A_k^S)$. This is a nonempty open set since $(g_v)_{v\notin S}$ lies in there. Since $G$ has strong approximation away from~$S$, there is $g\in \varphi^{-1}(V)\cap G(k)$ and so we obtain $g\cdot x\in V$, showing that $X(k)$ is dense in $\pr_S(X(\A_k)^{\br_1})$, that is, $\pr_S(X(\A_k)^{\br_1})\subseteq \overline{X(k)}^S$.
\par Then, by Theorem \ref{closed} and the natural inclusions $X(k)\subseteq \pr_S(X(\A_k)^{\br})\subseteq \pr_S(X(\A_k)^{\br_1})$,  we conclude that the equality $\overline{X(k)}^S=\pr_S(X(\A_k)^{\br_1})=\pr_S(X(\A_k)^{\br})$ holds.
\end{proof}


\subsection{Brauer-Manin pairing away from $S$ with $\br'_S X$}
To remedy the potential non-closedness of $\pr_S(X(\A_k)^{\br})$ for a general variety $X$, one can also define an \textit{a priori} closed subset in $X(\A_k^S)$ containing $X(k)$ using the Brauer-Manin obstruction, and hope to relate such a closed subset to the closure of $X(k)$ in $X(\A_k^S )$. For example, in \cite{harari2010brauer} and \cite{demeio2022etale}, the authors considered the subgroup $\br_S X=\ker(\br X\rightarrow \prod_{v\in S}\br X_{k_v})$ of elements which are trivial at places in $S$, and $\br_S X$ cuts out a closed subset $X(\A_k^S)^{\br_S}$ in $X(\A_k^S)$. We give examples in this section showing that $\overline{X(k)}^S=X(\A_k^S)^{\br_S}$ does not hold in general, when $S$ contains real places for homogeneous spaces $X$. We are thus motivated to use finer information of the Brauer-Manin obstruction, and we propose considering the subset $\br_S^{\prime} X$ (see \S\ref{notation}) of elements in $\br X$ which induce zero evaluation maps at places in $S$. 
\par Formally, consider the natural pairing
\begin{align*}
X(\A_k^S)\times \br^\prime_S X&\rightarrow \QZ \\
(\{x_v\},A)&\mapsto \sum_{v\in\Omega\backslash S}\inv_v(\ev_A (x_v)).
\end{align*}
 For a subset $B\subseteq \br_S^\prime X$, let $X(\A_k^S)^B$ be the set of $S$-adelic points which are orthogonal to $B$. It is easy to see that $X(\A_k^S)^B=\pr_S(X(\A_k)^B)$. The set $X(\A_k^S)^B$ is closed in $X(\A_k^S)$ by continuity of the pairing, and contains $X(k)$ by the Albert-Brauer-Hasse-Noether exact sequence. Therefore, when $X(\A_k^S)^B\neq X(\A_k^S)$, there is an obstruction to strong approximation away from $S$.
 \par When we take $B$ to be the whole of $\br_S^ \prime X$, we denote by $X(\A_k^S)^ {\br^ \prime_S}$ the corresponding Brauer-Manin set $X(\A_k^S)^B$. Similarly, we can consider $B$ to be $\br_S X$, $\br_{1,S}^\prime X$ or $\br_{1,S} X$. Clearly, we have chains of inclusions \begin{equation}
 \label{chainofinclusions}\begin{split}\overline{X(k)}^S\subseteq X(\A_k^S)^ {\br^ \prime_S}&\subseteq X(\A_k^S)^ {\br_S} \subseteq X(\A_k^S)^ {\br_{1,S}} \\
 \overline{X(k)}^S\subseteq X(\A_k^S)^ {\br^ \prime_S}&\subseteq X(\A_k^S)^ {\br^ \prime_{1,S}}.\end{split}\end{equation} In this section, we investigate conditions for when these inclusions become equalities for homogeneous spaces $X=G/H$ under semisimple simply connected groups $G$ with commutative stabilizers $H$, and give examples showing that these inclusions can be strict in general. As in the previous section, we use the decomposition $H\simeq M\times \Ga^r$ with $M$ of multiplicative type, and we identify $\Ho^1(F,H)$ with $\Ho^1(F,M)$ for any field $F$.
\begin{rmk}\label{moduloconstants}It is common practice to consider a Brauer-Manin pairing with the quotient group $\br X/\br k$ since constant elements do not contribute to the Brauer-Manin obstruction. For example, see \cite[\S7]{bright2015bad} where the local evaluation map associated to $\overline\alpha\in\br X/\br k$ is defined by fixing a ``base point'' $Q_v\in X(k_v)$: 
$$\overline\alpha\colon X(k_v)\rightarrow \br k_v ,\quad P_v\mapsto \alpha(P_v)-\alpha(Q_v)$$
for any representative $\alpha\in\br X $ of $\overline\alpha\in\br X /\br k$. With this point of view, we can define a Brauer-Manin pairing away from $S$ with $\br_{S}^\prime X/(\br k \cap \br_{S}^\prime X)=\br_{S}^\prime X/\br_S k.$ The natural map $\br_{S}^\prime X/\br_S k\to \{\beta\in \br X\,|\,\ev_{\beta}\text{ is constant at all }v\in S\}/\br k$ is an isomorphism because of the surjectivity of $\br k\rightarrow \prod_{v\in S}\br k_v$. Hence we may think of an element in $\br'_SX/\br_Sk$ as the class (modulo $\br k$) of an element of $\br X$ inducing constant evaluation maps at $S$. Similarly, there is an isomorphism $\br_S X/\br_S k\simeq \ker (\br X\rightarrow \prod_{v\in S}\br X_{k_v}/ \br k_v)/\br k$, the classes of elements which become constant at $S$, and this is the pairing considered by Harari and Voloch in \cite[\S2]{harari2010brauer}.
\end{rmk}
 \begin{lem}\label{discrete}
Let $\mu$ be a commutative étale group scheme over $k$ and let $S$ be a finite set of places of $k$. The map $r\colon\Ho^1(k,\mu)\to \prod'_{v\notin S}\Ho^1(k_v,\mu)$ has discrete image.
\begin{proof}
Let $y\in \Ho^1(k,\mu)$ and write $z\coloneq r(y)$ and $I\coloneq r(\Ho^1(k,\mu))$. Since the space $\prod'_{v\notin S}\Ho^1(k_v,\mu)$ is Hausdorff, it suffices to show that $z$ has a neighbourhood $U$ in $\prod'_{v\notin S}\Ho^1(k_v,\mu)$ that intersects $I$ in finitely many points. There exists a set $T$, disjoint from $S$, such that $T$ contains all infinite places not in $S$, and such that for all $v\notin T\cup S$ we have $r_v(y)\in \Ho^1(\Oh_v,\mu)$. We enlarge $T$ such that $\#\mu(\overline{k})$ is invertible in $\Oh_{T\cup S}$ and we set $R:=T\cup S$. Consider the open $U\coloneq \prod_{v\in T}\Ho^1(k_v,\mu)\times \prod_{v\notin R}\Ho^1(\Oh_v,\mu)$ of $z$.
For all nonarchimedean $v$ that do not divide $\#\mu(\overline{k})$, there is the Gysin sequence: \begin{equation*}0\to \Ho^1(\Oh_v,\mu)\to \Ho^1(k_v,\mu)\overset{\res}{\to} \Ho^0(\mathbb{F}_v,\mu(-1))\to ...\end{equation*}
So an element $x$ of $\Ho^1(k,\mu)$ maps into $U$ if and only if $\res(r_v(x))=0$ for all $v\notin R$. 
For finite sets $R\subseteq R'$ of places of $k$, containing all archimedean ones, there is an exact sequence:
\begin{equation*}0\to \Ho^1(\Oh_{R},\mu)\to \Ho^1(\Oh_{R'},\mu)
\xrightarrow[]{\res}\bigoplus_{v\in R'\setminus R}\Ho^0(\mathbb{F}_v,\mu(-1))\to ...\end{equation*}
Taking the colimit over all finite sets of places $R'$ containing $R$ gives the exact sequence:
\begin{equation*}0\to \Ho^1(\Oh_R,\mu)\to \Ho^1(k,\mu)\to \bigoplus_{v\in \Omega\setminus R}\Ho^0(\mathbb{F}_v,\mu(-1))\to ...\end{equation*}
Here the middle term is indeed $\Ho^1(k,\mathcal{F})=\Ho^1(\varprojlim_{R\subseteq R'}\Oh_{R'},\mathcal{F})=\varinjlim_{R\subseteq R'}\Ho^1(\Oh_{R'},\mathcal{F})$ where the second equality is from \cite[VII Théorème 5.7]{SGA4}. Hence we note that the elements $x\in \Ho^1(k,\mu)$ having residue zero at all $v\notin T\cup S$ can be identified with $\Ho^1(\Oh_{T\cup S},\mu)$. \par
We claim $\Ho^1(\Oh_R,\mu)$ is finite for any finite $R$ containing the archimedean places. After passing to a finite extension $L$ of $k$ we may assume that $\mu|_L$ is constant and isomorphic to a product of $\mu_n(L)$ (so we also adjoined roots of unity to $k$). Let $T$ be the set of places of $L$ extending $R$. The Kummer sequence for $\Oh_T$ reads: \[0\to\Oh_{T}^\times/(\Oh_{T}^{\times})^n\to \Het^1(\Oh_T,\mu_n)\to \Pic(\Oh_T)[n]\to 0.\] Combining Dirichlet's unit theorem with finiteness of the class group gives that $\Ho^1(\Oh_T,\mu_n)$ is finite and so $\Ho^1(\Oh_T,\mu)$ is finite. Now we want to descend the finiteness of this group to $\Oh_R$.
We have the exact sequence in group cohomology: 
\begin{equation*}0\to \Ho^1(\Gal(L/k),\mu(\overline{k}))\to \Ho^1(\Gamma_k,\mu(\overline{k}))\to \Ho^1(\Gamma_L,\mu(\overline{k}))^{\Gal(L/k)}\to...\end{equation*}
Note that $\ker(\Het^1(\Oh_R,\mu)\to \Het^1(\Oh_T,\mu))\subset \Ho^1(\Gal(L/k),\mu(\overline{k}))$ and the rightmost set is finite which follows from the cocycle description.
This implies that there is a surjection from the finite set $\Ho^1(\Oh_{R},\mu)$ to $I\cap U$ where $I$ is as in the start of the proof. This shows that any $z\in I$ has a finite open neighbourhood, which as discussed at the beginning completes the proof.
\end{proof}
\end{lem}

As a consequence of this lemma, we can generalize this statement to groups of multiplicative type.

\begin{cor}\label{multdiscrete}
For a group $M$ of multiplicative type, the image of $\Ho^1(k,M)$ inside ${\prod}^\prime_{v\notin S}\Ho^1(k_v,M)$ is discrete.
\end{cor}
\begin{proof}
Let $M^\circ$ be the neutral connected component of $M$ and let $\mu\coloneq \pi_0(M)$ be the étale group scheme of connected components. There is a commutative diagram where the rightmost column is an exact sequence of topological groups:
\begin{equation*}
\begin{tikzcd}
{\Ho^1(k,M^\circ)} \arrow[d] \arrow[r, "r"]          & {\bigoplus_{v\notin S}\Ho^1(k_v,M^\circ)} \arrow[d, "\alpha"] \\
{\Ho^1(k,M)} \arrow[r, "s"] \arrow[d] & {{\prod}^\prime_{v\notin S}\Ho^1(k_v,M)} \arrow[d, "\beta"]  \\
{\Ho^1(k,\mu)} \arrow[r, "t"]                           & {{\prod}^\prime_{v\notin S}\Ho^1(k_v,\mu)}.                  
\end{tikzcd}
\end{equation*}
We have to show that $\mathrm{Im}(s)$ is discrete. Let $\beta'\coloneq \beta|_{\mathrm{Im}(s)}\colon\mathrm{Im}(s)\to \mathrm{Im}(t)$. Pick any point $s(x)\in \mathrm{Im}(s)$ and note that $\{\beta'(s(x))\}$ is open in $\mathrm{Im}(t)$ by the previous lemma. Therefore $\beta'^{-1}(\beta(s(x)))$ is open in $\mathrm{Im}(s)$. Since we have $\ker(\beta')\subseteq \alpha(\bigoplus_{v\notin S}\Ho^1(k_v,M^\circ ))$, the kernel carries the discrete topology. So $\beta'^{-1}(\beta(s(x)))=s(x)+\ker(\beta')$ is discrete and we conclude that: $$\{s(x)\}\underset{\mathrm{open}}{\subset}\beta'^{-1}(\beta(s(x)))\underset{\mathrm{open}}{\subset}\mathrm{Im}(s).$$
\end{proof}

This gives the following immediate corollary.

\begin{cor}
The composed map $X(k)\to \Ho^1(k,M)\to \prod'_{v\notin S}\Ho^1(k_v,M)$ is locally constant for the subspace topology on $X(k)$ induced from $X(k)\subseteq X(\A_k^S)$.
\begin{proof}
The map is continuous by Proposition \ref{pointedsets}. By Corollary \ref{multdiscrete}, the image is discrete and thus the map is locally constant.
\end{proof}
\end{cor}

From this, we deduce the following useful lemma.

\begin{lem}\label{S-closure}
Suppose $G$ satisfies strong approximation away from $S$. Fix $(x_v)_{v\notin S}\in X(\A_k^S)$ with image $(y_v)_{v\notin S}\in {\prod}^\prime_{v\notin S}\Ho^1(k_v,M)$. The point $(x_v)_{v\notin S}$ lies in lies in $\overline{X(k)}^S$ if and only if $(y_v)_{v\notin S}=(r_v(y))_{v\notin S}$ for some $y\in \Ho^1(k,M)$ such that $r_v(y)$ maps to zero in $\Ho^1(k_v,G)$ for every $v\in S_\R$.
\begin{proof}
If a sequence $(x_n)_n\in X(k)^\mathbb{N}$ converges to $(x_v)_{v\notin S}$, the previous corollary tells us that for a sufficiently large $n$ we have that $\delta(x_n)\in \Ho^1(k,M)$ maps to $(y_v)_{v\notin S}\coloneq (\delta_v(x_v))_{v\notin S}$. For all $v\in S_\R$ we have $r_v(y)=r_v(\delta(x_n))$, which vanishes by Proposition \ref{pointedsets}.
\par Conversely suppose that $y\in \Ho^1(k,M)$ maps to $(y_v)_{v\notin S}$, such that $r_v(y)=0$ for all $v\in S_\R$. Since for all $v\in \Omega_\R\setminus S_\R$ there is the point $x_v\in X(k_v)$ mapping to $r_v(y)$, we obtain that $r_v(y)=0$ for all $v\in \Omega_\R$. By applying Theorem \ref{kneserharder} we obtain that the image of $y$ in $\Ho^1(k,G)$ is trivial and so $y=\delta(x)$ for some $x\in X(k)$. Since $x$ and $(x_v)_{v\notin S}$ have the same image in ${\prod}^\prime_{v\notin S}\Ho^1(k_v,M)$, there is by Proposition \ref{pointedsets} some $(g_v)_{v\notin S}$ such that $(x_v)_{v\notin S}=(g_v\cdot x)_{v\notin S}$. Since $G$ has strong approximation away from $S$, $(g_v)_{v\notin S}$ can be approximated by a sequence in $G(k)$, which gives the approximation of $(x_v)_{v\notin S}$ by a sequence in $X(k)$.
\end{proof}
\end{lem}

\par 
We recall the notation $\Ho_S^1(k,-)$ 
as defined in \S\ref{notation}. 

\begin{lem} \label{surjective}The natural map $\Ho^1_{S\setminus S_\R}(k,\hat{M})\to \prod_{v\in S_\R}\Ho^1(k_v,\hat{M})$ is surjective.
\begin{proof}The result of the special case $S=S_\R$ and $M$ is finite and commutative can be found in \cite[Lemmes 1.4 and 1.6]{Sansuc1981}. We now give a proof for the general case. The identification $\Ho^1(k_v,\hat{M})=\Ho^1(k_v,M)^*$ by $a\mapsto (f_a:x\mapsto x\cup a)$ induces a map $\prod^\prime \Ho^1(k_v,\hat{M})\to \Ho^1(k,M)^*$. Indeed this is just induced by the restriction map $\Ho^1(k,M)\to \prod'_v\Ho^1(k_v,M)$ and then taking duals on both sides. Since taking the dual is exact, the map $\Ho^1(k,M)\to \prod'_v\Ho^1(k_v,M)$ induces a surjection $\prod'_v\Ho^1(k_v,\hat{M})\twoheadrightarrow (H^1(k,M)/\Sha^1(k,M))^*$. It is an easy check to see that $\prod^\prime \Ho^1(k_v,\hat{M})\to \Ho^1(k,M)^*$ is given by $(a_v)_v\mapsto (x\mapsto \sum_v a_v\cup x)$ and its kernel is by \cite[Théorème 6.3]{demarche2011suites}  the image of $\Ho^1(k,\hat{M})\to \prod'_v\Ho^1(k_v,\hat{M})$. 
In short we have an exact sequence:
\[H^1(k,\hat{M})/\Sha^1(k,\hat{M})\hookrightarrow\prod_{v\in S_\R}H^1(k_v,\hat{M})\times \prod_{v\in S\backslash S_\R}H^1(k_v,\hat{M})\times {\prod_{v\notin S}}^\prime H^1(k_v,\hat{M})\twoheadrightarrow (H^1(k,M)/\Sha^1(k,M))^*\]
which gives the identification 
\begin{align*} \Ho^1_{S\setminus S_\R}(k,\hat{M})/\Sha^1(k,\hat{M})&=\Ker\left(\prod_{v\in S_\R}H^1(k_v,\hat{M})\times {\prod_{v\notin S}}^\prime H^1(k_v,\hat{M})\rightarrow (H^1(k,M)/\Sha^1(k,M))^*\right)\\
&=\Coker\left(H^1(k,M)/\Sha^1(k,M)\rightarrow \prod_{v\in S_\R}H^1(k_v,M)\times {\prod_{v\notin S}}^\prime H^1(k_v,M)\right)^*.\end{align*} Therefore, the surjectivity of $\Ho^1_{S\setminus S_\R}(k,\hat{M})/\Sha^1(k,\hat{M})\to \prod_{v\in S_\R}\Ho^1(k_v,\hat{M})$ is equivalent to the injectivity of the map below defined by the canonical inclusion $\prod_{v\in S_\R}\Ho^1(k_v,M)\to \prod^\prime \Ho^1(k_v,M)$:
\[\prod_{v\in S_\R}\Ho^1(k_v,M)\rightarrow \Coker\left(H^1(k,M)/\Sha^1(k,M)\rightarrow \prod_{v\in S_\R}H^1(k_v,M)\times {\prod_{v\notin S}}^\prime H^1(k_v,M)\right).\] Now we prove this injectivity.  Take a totally imaginary Galois extension $L/k$ such that $L$ splits $M$ and denote for a place $w$ of $L$ by $D_w$ the decomposition group of $w$. For every place $v\in S_\R$ with $w\in \Omega_L$ extending $v$, we have $\Ho^1(k_v,M)=\Ho^1(D_w,M(L))$. The Chebotarev density theorem implies that for all $v\in S_\R$ there is a place $v'\notin S$ with an extension $w'$ to $L$ such that $D_{w'}=D_{w}$. Hence we have: \[\Ho^1(D_w,M(L))=\Ho^1(D_{w'},M(L))\hookrightarrow \Ho^1(k_{v'},M).\]
Here the last arrow comes from the inflation-restriction sequence. This implies that if an element in $\Ho^1(k,M)$ vanishes in $\Ho^1(k_v,M)$ for all $v\notin S$, then it also vanishes in $\Ho^1(k_v,M)$ for all $v\in S_\R$. Thus this finishes the proof of the claim, which is equivalent to the statement in the proposition.
\end{proof}
\end{lem}

Using the above lemma, we can understand when an adelic point is orthogonal to $\br_{1,S}'X$ or~$\br_{1,S}X$.

\begin{lem}\label{orthogonaltoBr}
Let $r\colon\Ho^1(k,M)\to \prod_{v\notin S}\Ho^1(k_v,M)$ be the restriction map, and let $\langle\delta_v(X(k_v))\rangle$ be the subgroup generated by $\delta_v(X(k_v))$ inside $\Ho^1(k_v,M)$. 
\begin{itemize}
\item A point $(x_v)_{v\notin S}\in X(\A_k^S)$ is orthogonal to~$\br_{1,S}X$ if and only if $(\delta_v(x_v))_{v\notin S}$ lies in $r(\Ho^1(k,M))$.
\item A point $(x_v)_{v\notin S}\in X(\A_k^S)$ is orthogonal to $\br_{1,S}'X$ if and only if $(\delta_v(x_v))_{v\notin S}=(r_v(y))_{v\notin S}$ for some $y\in \Ho^1(k,M)$ such that $y_v:=r_v(y)\in \langle \delta_v(X(k_v))\rangle$ for all $v\in S_\R$.
\end{itemize}
\begin{proof}
We use the commutative diagram (\ref{Poitou-Tate}) with exact rows. For the first statement: Let $(x_v)_{v\notin S}\in X(\A_k^S)$ be orthogonal to $\br_{1,S} X$. By the diagram (\ref{br1pairing}) this is equivalent to $(\delta_v(x_v))_{v\notin S}$ being orthogonal to: $$\{z\in \Ho^1(k,\hat{M})\,|\,-\cup z\colon\Ho^1(k_v,M)\to \Q/\Z\text{ is zero for }v\in S\}.$$
By Poitou-Tate duality (\ref{PTpairing}) we conclude that $(\delta_v(x_v))_{v\notin S}$ goes to zero in $\Ho^1_S(k,\hat{M})^*$ and thus $(\delta_v(x_v))_{v\notin S}\in r(\Ho^1(k,M))$ by diagram (\ref{Poitou-Tate}). Note that the argument reverses which proves the first statement.
\par For the second statement: The adelic point $(x_v)_{v\notin S}$ is orthogonal to $\br_{1,S}'X$ if and only if we have that $(y_v)_{v\notin S}\coloneq (\delta_v(x_v))_{v\notin S}$ is orthogonal to: \begin{equation}\label{H1awayfromS}\{z\in \Ho^1(k,\hat{M})\,|\,-\cup z\colon\delta_v(X(k_v))\to \Q/\Z\text{ is zero for }v\in S\}.\end{equation}
For a finite place $v$, we have that Poitou-Tate duality (\ref{PTpairing}) combined with surjectivity (which is implied by Theorem \ref{kneserharder}) of $X(k_v)\to \Ho^1(k_v,M)$ gives that the $z\in \Ho^1(k,\hat{M})$ being orthogonal to $\delta_v(X(k_v))$ under the natural pairing is equivalent to $r_v(z)=0$. 
\par For one direction, suppose that $(y_v)_{v\notin S}=r(y)$ for some $y\in \Ho^1(k,M)$, which satisfies $r_v(y)\in \langle\delta_v( X(k_v))\rangle$ for all $v\in S_\R$. Then we have for all $z$ in the set (\ref{H1awayfromS}):
\begin{equation}\label{suminvariants}\sum_{v\notin S}y_v\cup r_v(z)=\sum_{v\notin S}r_v(y\cup z)=-\sum_{v\in S} y_v\cup r_v(z)=-\sum_{v\in S_\R}y_v\cup r_v(z).\end{equation}
The rightmost sum vanishes by the assumption on $y$ as in the statement of the Lemma combined with that $z$ lives in the set (\ref{H1awayfromS}). So indeed $(y_v)_{v\notin S}$ is orthogonal to the set (\ref{H1awayfromS}) in this case.
\par Conversely, if $(y_v)_{v\notin S}$ is orthogonal to the set (\ref{H1awayfromS}), then it comes from some $y\in \Ho^1(k,M)$ by the proof of the first statement. So by the same calculation as (\ref{suminvariants}), we obtain that for any $z$ in the set (\ref{H1awayfromS}), the pairing is given by $\langle (y_v)_{v\notin S},z\rangle = -\sum_{v\in S_\R}r_v(y)\cup r_v(z)$. Suppose that there exists $w\in S_\R$ such that $r_w(y)\notin \langle \delta_w(X(k_w))\rangle $. By Lemma \ref{surjective}, we have a surjection: $$\Ho^1_{S\setminus S_\R}(k,\hat{M})\to \prod_{v\in \Omega_\R} \Ho^1(k_v,\hat{M}).$$ 
Since $\langle \delta_w(X(k_w))\rangle$ is a linear subspace of $\Ho^1(k_w,M)$, we may pick some $z_w\in \Ho^1(k_w,\hat{M})$ such that $z_w$ gives the zero map on $\langle \delta_w(X(k_w))\rangle$ and such that $z_w\cup y_w\neq 0$. Pick $z\in \Ho^1_{S\setminus S_\R}(k,\hat{M})$ such that $r_w(z)=z_w$ and such that for all other $v\in S_\R$ we have $r_v(z)=0$. Then note that $z$ lies in the set (\ref{H1awayfromS}) which contradicts that by our construction we have $\langle (y_v)_{v\notin S},z\rangle\neq 0$. So we conclude that the second statement holds.
\end{proof}
\end{lem}

We recall the notation $\Sha^1_S(k,-)$ as defined in \S \ref{notation}.

\begin{thm}\label{main3}
Suppose $G$ satisfies strong approximation away from $S$. We have

\begin{enumerate}[label=(\alph*)]
\item $X(\A_k^S)^{\br_{1,S}}=\overline{X(k)}^S$ if and only if for $r\coloneq (r_v)_{v\in S_\R}$ there is an equality: \begin{equation}\label{brauer1}\prod_{v\in S_\R}\Ho^1(k_v,M)=r(\Sha^1_S(k,M))+\left[\prod_{v\in S_\R}\delta_v(X(k_v))\right]\tag{*}\end{equation}
 \item $X(\A_k^S)^{\br'_{1,S}}=\overline{X(k)}^S$ if and only if for $r\coloneq (r_v)_{v\in S_\R}$ there is an inclusion:  \begin{equation}\label{brauer1'}\prod_{v\in S_\R}\langle\delta_v(X(k_v))\rangle\subseteq r(\Sha^1_S(k,M))+\left[\prod_{v\in S_\R}\delta_v(X(k_v))\right].\tag{**}\end{equation}\end{enumerate}
\begin{proof}
Suppose that the equation (\ref{brauer1}) holds and pick any $(x_v)_{v\notin S}$ that is orthogonal to $\br_{1,S} X$ and write $y_v\coloneq \delta_v(x_v)$. By Lemma \ref{S-closure}, it will suffice to show that there exists an $z\in \Ho^1(k,M)$ such that $(r_v(z))_{v\notin S}=(y_v)_{v\notin S}$ and such that $r_w(z)\in \delta_w(X(k_w))$ for all $w\in S_\R$. It follows from Lemma \ref{orthogonaltoBr} that there is $y\in \Ho^1(k,M)$ such that $r_v(y)=y_v$ for $v\notin S$ and by equation (\ref{brauer1}) we can add an element of $\Sha^1_S(k,M)$ to $z$ such that $r_w(z)\in \delta_w(X(k_w))$ for all $w\in S_\R$. 
\par Conversely, suppose that there is some $(y_w)_{w\in S_\R}\in \prod_{w\in S_\R}\Ho^1(k_w,M)$ that is not the sum of an element in $r(\Sha^1_S(k,M))$ with an element in $\prod_{w\in S_\R}\delta_w(X(k_w))$. By Lemma \ref{surjective}, the map $\Ho^1(k,M)\to \bigoplus_{v\in \Omega_\R}\Ho^1(k_v,M)$ is surjective. Pick $y\in \Ho^1(k,M)$ such that $r_v(y)=0$ for $v\in \Omega_\R\setminus S_\R$ and $r_v(y)=y_v$ for $v\in S_\R$. By combining Proposition \ref{pointedsets} and Theorem \ref{kneserharder} we see that $(r_v(y))_{v\notin S}$ lifts to a point $(x_v)_{v\notin S}$, is orthogonal to $\br_{1,S}X$ by Lemma \ref{orthogonaltoBr}. However $(x_v)_{v\notin S}$ does not lie in $\overline{X(k)}^S$ by Lemma \ref{S-closure}.
\par The proof of the second statement is analogous and also follows from Lemma \ref{orthogonaltoBr}.
\end{proof}
\end{thm}

\subsection{Example of $\overline{X(k)}^S=X(\A_k^S)^{\br^\prime_{S}}\subsetneq X(\A_k^S)^{\br_{S}}$}
As promised, we give examples showing the necessity of considering $\br^\prime_S X$ (versus just its subgroup $\br_S X$) to cut out the strong approximation locus in certain cases. We first give an example with finite commutative stabilizer.

\begin{prop}\label{egSO}
Set $X=\SO_n$ and set $G=\Spin_n$ which are varieties over $k=
\QQ$. Then $X$ is a homogeneous space under $G$ with stabilizer $\mu_2$. Take $S=\{\infty,v_0\}$ where $v_0$ is a finite place. Then $$\overline{X(k)}^S=X(\A_k^S)^{\br^\prime_{S}}\subsetneq X(\A_k^S)^{\br_{S}}.$$
\end{prop}
\begin{proof}
We have 
$$\br X/\br k=\br_1 X/\br k\simeq \Ho^1(k,\mu_2)\simeq k^\times/k^{\times2}.$$ The map $G(\RR)\rightarrow X(\RR)$ is surjective since $\SO_n(\R)$ is connected, so the map $X(\RR)\rightarrow \Ho^1(\RR,\mu_2)$ has trivial image. Therefore, every element in $\br X_\RR$ induces the trivial evaluation map on $X(\RR)$. We know that $G$ satisfies strong approximation away from $S$ since $G_{v_0}$ is not compact. Let $(a_v)_{v\notin S}\in\prod_{v
\notin S}^\prime \Ho^1(\QQ_v,\mu_2)$ be the image of $-1\in \QQ^\times/\QQ^{\times 2}\simeq \Ho^1(\QQ,\mu_2)$ under the diagonal map. Lift $a_v$ to a local point $P_v\in X(\QQ_v)$ which is possible by the vanishing of $\Ho^1(\QQ_v,G)$, and $(P_v)_{v\notin S}$ is actually inside $X(\A_k^S)$ by Proposition \ref{pointedsets}. Then by Lemma \ref{orthogonaltoBr}, we see that $(P_v)_{v\notin S}\in X(\A_k^S)$ is orthogonal to $\br_S X$ but not orthogonal to $\br_S^\prime X$. Finally, since $\mu$ is central in $G$, the condition (\ref{brauer1'}) in Theorem \ref{main3} is satisfied and we get $\overline{X(k)}^S=X(\A_k^S)^{\br^\prime_{1,S}}.$
\end{proof}

Then we give another example with toric stabilizer. 
\begin{prop}\label{egSpin/T}
Let $q(x,y,z)$ be a quadratic form in 3 variables over $\Q$. Consider $X/\Q$ a smooth affine quadric defined by an equation $q(x,y,z)=a$, which is a homogeneous space under $\Spin (q)$ with stabilizer a one-dimensional torus. Assume $X$ has a rational point. Let $d\coloneq -a\det q$. Suppose $q$ is anisotropic over $\R$ and $d\notin \R^ {\times 2}$ (i.e. $d<0$), but there is a finite place $v$ such that $q$ is isotropic over $k_v$ and $d\in k_v^ {\times 2}.$ Let $S=\{\infty, v\}$. Then $$\overline{X(k)}^S=X(\A_k^S)^{\br^\prime_{S}}\subsetneq X(\A_k^S)^{\br_{S}}.$$
\end{prop}
\begin{proof}
We have $\br X_{\RR}/\br \RR\simeq \br X /\br \QQ\simeq \ZZ/2\ZZ$ by \cite[Proposition 4.1 and Lemma 4.4]{colliot2013strong}. Hence $\br_S X/\br_S \Q=0$ and $X(\A_{\Q}^S)^{\br_S}=X(\A_{\Q}^S).$ However, by \cite[Lemma~4.4]{colliot2013strong} and Remark \ref{moduloconstants}, we have that $$\br_S^{\prime} X/\br_ S \Q=\{\overline\alpha\in\br X/\br \Q|\alpha^ *\colon X(\Q_v)\rightarrow \br\Q_v\text{ is constant for all } v\in S\}\simeq\ZZ/2\ZZ,$$ and that there exists a place $v_1\notin S$ such that $\xi$ takes distinct values over $X(\Q_{v_1})$ where $\xi\in \br X$ has nonzero image in $\br_S^{\prime} X/\br_S\Q $. Therefore, we have $X(\A_{\Q}^S)^{\br^\prime_S}\subsetneq X(\A_{\Q}^S),$ giving an obstruction to strong approximation away from $S$. By \cite[Proposition ~4.5]{colliot2013strong} and Theorem \ref{closed}, we know that $\overline{X(\Q)}^S={\pr_S(X(\A_{\Q})^{\br}}$, but $\br X/\br \Q=\br_S X/\br_S \Q$ cut out the same Brauer-Manin set, so we have $\pr_S(X(\A_{\Q})^{\br})={\pr_S(X(\A_{\Q})^{\br^ {\prime}})}=X(\A_{\Q}^S)^{\br^{
\prime}_S}=\overline{X(\Q)}^S$. 
\end{proof}
\begin{eg}
Let $q$ be the quadratic form $x^2+y^2+z^2$ over $\Q$. Clearly it is anisotropic over $\R$, but isotropic over $\Q_5$ since we have the nontrivial zero $(0,i,1)$ of $q$ where $i^2=-1$. Notice that the discriminant of $q$ equals $1$. Now take the equation $q(x,y,z)=1$ defining our smooth affine quadric $X/\Q$, which has a rational point $(1,0,0)$. Since $d=-\det(q)=-1$, which is not a square in $\R$ but a square in $\Q_5$, we can apply the above proposition to $S=\{\infty, 5\}$ and get $\overline{X(\Q
)}^S={\pr_S(X(\A_
{\Q})^{\br}})= X(\A_
{\Q}^S)^{\br^\prime_S}\subsetneq X(\A_
{\Q}^S)^{\br_S}=X(\A_
{\Q}^S).$
\end{eg}

\subsection{Example of $\overline{X(k)}^S=X(\A_k^S)^{\br^\prime_{S}}\subsetneq X(\A_k^S)^{\br^\prime_{1,S}}$}
In this section, we show that transcendental elements in $\br^\prime_S X$ can play a role in cutting out the precise strong approximation locus.

\begin{prop}{\cite[Proposition~3.4]{arteche2018unramified}}\label{extended pairing}
Set $X=G/\mu$ where $\mu$ is a finite group scheme. Let $K$ be a field extension of $k$. There is a pairing $\Ho^1(K,\mu)\times \br{X}\to \br K $ such that the following diagram commutes:\begin{equation}
\begin{tikzcd}
X(K) \arrow[d] \arrow[r,phantom,"\times" description] & \br X \arrow[r] \arrow[d,Rightarrow, no head]      & \br K \arrow[d,Rightarrow, no head] \\
\Ho^1(K,\mu)                                \arrow[r,phantom,"\times" description] & {\br X} \arrow[r]  & \br K.     
\end{tikzcd}
\end{equation}
\end{prop}

\begin{rmk}
We can also extend the adelic pairing $X(\A_k)\times \br X\to \Q/\Z$ to a pairing ${\prod}^\prime \Ho^1(k_v,\mu)\times \br X \to \Q/\Z$. To see this, observe that by the combination of Proposition \ref{pointedsets} and Theorem \ref{kneserharder}, the map ${\prod}^\prime_{v\notin \Omega_\R}X(k_v)\to {\prod}^\prime_{v\notin \Omega_\R}\Ho^1(k_v,\mu)$ is surjective. Hence by applying Proposition \ref{extended pairing}, we see that for any $(y_v)_v\in {\prod}^\prime \Ho^1(k_v,\mu)$ and for all but finitely many $v$ and any $\beta\in \br X$, the quantity $\langle y_v,\beta\rangle$ is zero. This implies that the pairing ${\prod}^\prime \Ho^1(k_v,\mu)\times \br X \to \Q/\Z$ is well defined.
\end{rmk}

\begin{rmk}\label{descibedpairing}
The pairing in Proposition \ref{extended pairing} can be described as follows. We denote by $X_K$ the base change of $X$ to $K$ (for example, $K=k_v$). We consider the image of $\br X $ in $\br X_K$ and define the pairing $X(K)\times \br X_K\rightarrow \br K$. A $K$-point $o\in X_K(K)$ gives rise to an identification $\pi_1(X_K)=\mu(\overline{K})\rtimes \Gamma_{K}$. The Hochschild-Serre spectral sequence attached to $\overline{G_K}\to X_K$ gives $\br X_K\cong \Ho^2(\pi_1(X_K),\overline{K}^\times)$. The general framework of group cohomology implies that there is a one-to-one correspondence between elements in $\Ho^1(K,\mu)$ and continuous sections $\Gamma_{K}\to \pi_1(X_K)$ taken up to conjugation by $\pi_1(\overline{X_K})$. It is induced by sending a cocycle $y\in \Ho^1(K,\mu)$ to the section $s_y:\gamma\mapsto (y(\gamma),\gamma)\in \mu(\overline{K})\rtimes \Gamma_{K}=\pi_1(X_K)$. So an element $y$ of $\Ho^1(K,\mu)$ defines a pullback $s_y^*\colon\Ho^2(\pi_1(X_K),\overline{K}^\times)\to \Ho^2(K,\overline{K}^\times)=\br K$.
\end{rmk}

For the examples that we will construct, we work with the following algebraic group over $\Q$.

\begin{df}\label{SUpq}
Set $L\coloneq \Q[i]$ and set $J_{p,q}$ to be the matrix $\diag(\underbrace{-1,...,-1}_{p\text{ times}},\underbrace{1,...,1}_{q\text{ times}})$. Denote for $p+q=\colon n$ by $\SU_{p,q}$ the group functor $\mathrm{Alg}/\Q\to \mathrm{Grp}$ defined by: 
\begin{equation*}R\mapsto \{m\in \SL_{n}(L\otimes_\Q R)\,|\,\overline{m}^tJ_{p,q}m=J_{p,q}\}.\end{equation*}
\end{df}

\begin{rmk}Since the condition $\overline{m}^tJ_{p,q}m=J_{p,q}$ is a condition given by the vanishing of polynomials over $\Q$, the group functor $\SU_{p,q}$ is an algebraic subgroup of the Weil restriction $R_{L/\Q}(\SL_{n})$.\par
We explain now why $\SU_{p,q}$ becomes isomorphic to $\SL_n$ after a base-change to $L$. There is an involution of the second kind (as in \cite{bookofinvolutions}) $\tau:\Mat_n(L)\to \Mat_n(L)\opp$ given by $m\mapsto J_{p,q}\overline{m}^tJ_{p,q}$. Denote for any $\Q$-algebra $R$ the base-change of $\tau$ to $R$ by $\tau_R$, which is an involution $\tau_R:\Mat_n(L)\otimes_\Q R\to \Mat_n(L)\opp\otimes_\Q R$. By using the canonical isomorphism $\Mat_n(L)\otimes_\Q R\cong \Mat_n(L\otimes_\Q R)$, we see that the functor $\SU_{p,q}$ is isomorphic to the functor: \[R\mapsto \{m\in \Mat_n(L)\otimes_\Q R\,|\,\tau_R(m)\cdot m=\Id,\,\mathrm{Nrd}(m)=1\}.\]
After base-change to $L$, this functor becomes isomorphic to $\SL(\Mat_n(L))$ (see \cite[p. 346]{bookofinvolutions}). So $\SU_{p,q}$ is a semisimple and simply connected algebraic group over $\Q$.
It is well known that as soon as $p,q\neq 0$, the real lie group $\SU_{p,q}(\R)$ is non-compact. Hence $\SU_{p,q}$ has strong approximation away from $\{\infty\}$ by \cite[Theorem~7.21]{platonov1993algebraic}.
\end{rmk}

\begin{df}\label{diagonalT} Let $T$ be the diagonal maximal torus of $\SU_{p,q}$ and let  $T[m]$ be the algebraic subgroup of $m$-torsion points; $T[m](R)=\{\diag(a_1,...,a_{p+q-1},\prod_i a_i^{-1})\,|\,a_i\in \mu_n(R\otimes_\Q \Q[i])\}$. 
\end{df}

We remind the reader of the norm $1$ subgroups of Weil restrictions of commutative algebraic groups. Associated to a finite Galois extension $M/K$ and a commutative algebraic group $H$, the norm $1$ algebraic group $R^1_{M/K}H$ is by definition  $\ker(N_{M/K}:R_{M/K}(H)\to H)$. 
\par The maximal torus $T$ and its torsion comes with the following properties.

\begin{lem}\label{diagonalTproperties}
The following statements about $T$ hold.
\begin{itemize}
    \item As algebraic groups over $\Q$, $T$ is isomorphic to $(R^1_{L/\Q}\G_m)^2$ and $T[m]$ is isomorphic to $(R^1_{L/\Q}(\mu_m))^2$. The group scheme $T[m]$ is constant for $m=4$. After base-change to $\R$, the group scheme $T[m]_\R$ is constant for any $m$.
    \item For any field $M$, the Galois cohomology of $T$ is given by the formula: \[\Ho^1(M,T)\cong \left(M^\times/\mathrm{Im}[N_{M(i)/M}:M(i)^{\times}\to M^\times]\right)^2\]
    \end{itemize}
    \begin{proof}
For the first part, there is a map $(R^1_{L/\Q}\G_m)^2\to T$ defined by $(a,b)\mapsto \diag(a,b,\frac{1}{ab})$ on points. One checks easily that the determinant condition on $T\subset \SU_{p,q}$ implies that the map is an isomorphism. The restriction of the above map to $(R^1_{L/\Q}(\mu_m))^2$ yields an isomorphism onto $T[m]$. The group scheme $T[4]\cong (R^1_{L/\Q}(\mu_4))^2$ has $16$ points over $\overline{\Q}$. But $R^1_{L/\Q}(\mu_4)$ already has $4$ points $\{\pm 1,\pm i\}$ over $\Q$, so $T[4](\Q)$ contains all geometric points and hence $T[4]$ is constant. A similar argument shows that after base-change to $\R$, $T[m]$ is constant since we have the isomorphism $R^1_{\C/\R}(\mu_m)(\R)\cong \Z/m$.
\par For the second part, it suffices to show that $\Ho^1(M,R^{1}_{L/\Q}\G_m)=M^\times/N_{M(i)/M}(M(i)^\times)$. This comes from the exact sequence of tori 
\[1\rightarrow R^1
_{M(i)/M}({\Gm}_{,M(i)}) \rightarrow R_{M(i)/M}({\Gm}_{,M(i)}) \rightarrow {\Gm}_{,M} \rightarrow 1\] (see \cite[(7.8)]{colliot2021brauer}). 
    \end{proof}
\end{lem}

\begin{rmk}\label{H1oftorsion}
As a bonus for $T[m]_\R$ being constant, we obtain: \[\Ho^1(\R,T[m])=\Hom(\Z/2,T[m](\C))\cong \begin{cases} \{\diag(a_1,...,a_{n-1},\prod_{i=1}^{n-1}a_i)\,|\,a_i\in \pm 1\}\text{ if }m\text{ is even}\\
0\text{ else.}\end{cases}\]
\end{rmk}

To understand the map $\Ho^1(\R,T[m])\to \Ho^1(\R,\SU_{p,q})$ we have the following interpretation of $\Ho^1(\R,\SU_{p,q})$.  

\begin{lem}\label{realhermitianform}
Let $h\colon\C^n\times \C^n\to \C$ be the hermitian form defined by $J_{p,q}$. Denote by $\xi\colon\C\to \C$ the complex conjugation. There is a bijection: \begin{equation*}\Ho^1(\R,\SU_{p,q})\leftrightarrow \{\R\text{-isomorphism classes of nonsingular hermitian forms }h'\text{ with }\disc(h)=\disc(h')\}~.\end{equation*}
The distinguished element in $\Ho^1(\R,\SU_{p,q})$ corresponds to the class with signature $(p,q)$. If a cocycle is given by $\xi\mapsto m$ for $m\in \SU_{p,q}(\R)$, the resulting class of hermitian forms is represented by $J_{p,q}\cdot m$.
\end{lem}

This result should be well known in the literature (cf. \cite[29.19]{bookofinvolutions} and \cite{allenhermitianforms}), but we included a self-contained proof for the non-expert which can be found in the Appendix. We thank Mikhail Borovoi for helping to explain to us such a description along with the example in the following proposition.
\par 
Note that over $\R$, the signature of a hermitian form $h'$ determines its isomorphism class. In the following example we see that although $\br_{1,S}'X$ may not cut out $\overline{X(k)}^S$, it might still be cut out by $\br'_S X$. By abuse of notation, we write $\infty$ for the set $\{\infty\}$ below.

\begin{prop}\label{egSU/T2}
We have for $G=\SU_{2,1}$ over $
k=\Q$, $\mu=T[2]$ and $X=G/\mu$ the inclusions: \begin{equation*}\overline{X(\Q)}^{\infty}= X(\A_\Q^\infty)^{\br'_{\infty}}\subsetneq X(\A_\Q^\infty)^{\br'_{1,\infty}}.\end{equation*}
\begin{proof}
To prove the strict inclusion $\overline{X(\Q)}^\infty\subsetneq X(\A_\Q^\infty)^{\br'_{1,\infty}}$: We have the equalities \[T[2](\Q)=\{\diag(a_1,a_2,a_1\cdot a_2)\,|\,a_i\in \{\pm 1\}\}=T[2](\overline{\Q}).\] So $T[2]$ is isomorphic to $(\Z/2)^2$. Therefore \cite[Lemme~1.1]{Sansuc1981} tells us that we have $\Sha^1_{\infty}(k,T[2])=~0$. Therefore Theorem~\ref{main3} gives that the inclusion is an equality if and only if we have that $\delta_\R(X(\R))\subseteq \Ho^1(\R,T[2])$ is a subgroup. We have $\Ho^1(\R,T[2])=\{\diag(a_1,a_2,a_1\cdot a_2)\,|\,a_i\in \{\pm 1\}\}$ and the map to $\Ho^1(\R,\SU_{2,1})$ is by the above lemma given by \[T[2]\ni m\mapsto [J_{2,1}\cdot m]\in \Ho^1(\R,\SU_{2,1}).\] We have that $\delta_\R(X(\R))\subseteq \Ho^1(\R,T[2])$ corresponds to the elements of $T[2]$ which map to hermitian forms with signature $(p,q)$. These are $\{(1,1,1),(-1,1,-1), (1,-1,-1)\}$, which do not form a subgroup of $\Ho^1(\R,T[2])$.
\par So we conclude that $X(\Q)^{\infty}\subsetneq (\A_\Q^{\infty})^{\br_{1,\infty}'}.$\\

In fact, the only cocycle that does not come from $\delta_\R(X(\R))$ is represented by $\diag(-1,-1,1)$. Denote by $\iota\colon\mu_2\to \overline{k}^\times$ the canonical inclusion. Since $\Gamma_k$ acts trivially on $T[2]$, we obtain that $\pi_1(X)$ (resp. $\pi_1(X_\R)$) is identified with $T[2]\times \Gamma_k$ (resp. $T[2]\times \Gamma_\R)$ (for this identification we use the section $o_*:\Gamma_k\to \pi_1(X)$ (resp. $o_*:\Gamma_\R\to \pi_1(X_\R)$)). Hence projection on the first factor gives us a retraction $\lambda_o:\pi_1(X_\R)\to T[2]$ (resp. $\pi_1(X)\to T[2]$). We consider the homomorphism $\iota_*\lambda_o^*\colon\Ho^2(T[2],\mu_2)\to \Ho^2(\pi_1(X_\R),\C^{\times})$, which factors via $\Ho^2(\pi_1(X),\overline{k}^\times)$. We also consider a section $\sigma:\Gamma_\R\to \pi_1(X_\R)$ Since the pushforwards and pullbacks below commute (this is easily seen from how pullbacks and pushforwards act on cocycles), there is a commuting diagram:

\begin{center}
\begin{tikzcd}
                                               & {\Ho^2(T[2],\mu_2)} \arrow[d, "\iota_*\lambda_o^*"'] \arrow[rd] \arrow[ld, "(\lambda_o\circ s)^*"'] &                                                 \\
{\Ho^2(\R,\mu_2)} \arrow[rd, "\iota_*"] & {\Ho^2(\pi_1(X_\R),\C^{\times})} \arrow[d, "\sigma^*"]                             & {\Ho^2(\pi_1(X),\overline{k}^\times)} \arrow[l] \\
                                               & {\Ho^2(\R,\C^\times).}                                                            &                                                
\end{tikzcd}
\end{center}

Note that the map $\iota_*$ in the bottom of the diagram is an isomorphism.
We have $T[2]\cong (\Z/2)^2$ and thus we get $\Ho^2(T[2],\mu_2)\cong \Ext_c((\Z/2)^2,\Z/2)$. We pick an isoomorphism $T[2]\to (\Z/2)^2$ such that $\diag(-1,-1,1)$ is sent to $(1,0)$ and from now on we use this identification. Our cocycle of interest is in this case \[f\colon\Gamma_\R\to (\Z/2)^2\quad \xi\mapsto (1,0).\] Consider the group $D_8=\langle r,s\,|\,r^4,s^2,rsrs \rangle$. Consider the element of $\Ext_c((\Z/2)^2,\Z/2)$, which is $\alpha=[1\to \Z/2\to D_8\to (\Z/2)^2\to 1]$, where $D_{8}\to (\Z/2)^2$ is the quotient map modulo $\langle r^2\rangle$ and $(1,0)=\overline{r}$ and $(0,1)=\overline{s}$. A short case-by-case calculation shows that pulling back $\alpha$ by any $g\in \Ho^1(\R,T[2])$ that is distinct from $f$ gives a split extension, while pulling back via $f$ yields a non-split extension (namely $\Z/4$).\par
Now consider the image of $\alpha$ in $\Ho^2(\pi_1(X),\overline{k}^\times)=\br{X}$ (see \cite[Proposition 3.2]{arteche2018unramified}). It lies in $\br'_{\infty}(X)$ since it pairs trivially with $\delta_\R(X(\R))$ by the end of the previous paragraph. For a point $(x_v)_{v\neq\infty}\in X(\A_\Q^\infty)^{\br'_{\infty}}$ we have that the image $(y_v)_{v\neq\infty}\in {\prod}^\prime_{v\neq\infty} \Ho^1(\Q_v,T[2])$ comes from a unique $y\in \Ho^1(\Q,T[2])$ by the vanishing of $\Sha^1_\infty(\Q,T[2])$. Since all local invariants sum up to~$0$, we have $\langle \alpha, (x_v)_{v\neq\infty}\rangle = \langle r_\R(\alpha),r_\R(y)\rangle$ which vanishes if and only if $r_\R(y)\neq \diag(-1,-1,1)$ by the end of the previous paragraph. So $(x_v)_{v\neq\infty}$ lies in $\overline{X(\Q)}^{\infty}$ by Lemma \ref{S-closure}.
\end{proof}
\end{prop}

\begin{rmk}It is easy to see that the Brauer element $\alpha$ that was constructed above is transcendental since its evaluation map $(\Z/2)^2\cong \Ho^1(\R,T[2])\to \Q/\Z$ is not linear: Three elements go to $0$ and one goes to~$\frac{1}{2}$. However any element in $\br_1 X\cong \Ho^1(k,T[2])\times \br k $ gives a linear map induced by the cup-product in Galois cohomology.
\end{rmk}

\subsection{Example of $\overline{X(k)}^S\subsetneq X(\A_k^S)^{\br^\prime_{S}}$}
It is natural to ask whether $\overline{X(k)}^S=X(\A_k^S)^{\br'_S}$ always occurs. We establish examples to show that this is not always the case. First we give an example in which the stabilizer is a torus.

\begin{prop}\label{extorus}
Let 
 $G=\SU_{p,q}$ over $\Q$ with $p,q\neq 0$ and $p+q\geq 3$. Let $T$ be the diagonal maximal torus of $G$ as defined in Definition \ref{diagonalT} and let $X\coloneq \SU_{p,q}/T$. Then we have the strict inclusion $$\overline{X(\Q)}^{\infty}\subsetneq~X(\A_\Q^{\infty})^{\br'_{\infty}}.$$
\begin{proof}
We first show that $\Sha^1_{\infty}(\Q,T)=0$. Since $T\cong (R^{1}_{L/\Q}\G_m)^2$, it suffices to show that $\Sha^1_{\infty}(\Q,R^{1}_{L/\Q}\G_m)=0$. Using the identification from Lemma \ref{diagonalTproperties}, we take one element in $\Sha^1_{\infty}(\Q,R^1_{L/\Q}\Gm)$ and it corresponds to $a\in \Q^\times$ which is a norm from $\Q_v(i)$ for all places $v\neq\infty$. This means that the quaternion algebra $(-1,a)\in \br \Q[2] $ vanishes in $\br \Q_v $ for all $v\neq\infty$. The exact sequence $$0\to \br \Q \to \bigoplus_v \br \Q_v \xrightarrow{\sum_v \inv_v} \Q/\Z\to 0$$ implies that $(-1,a)$ vanishes in $\br \R$ and hence $(-1,a)=0$. The vanishing of $(-1,a)\in\br \Q[2]$ is equivalent to $a$ being a norm from $\Q(i)$, and thus $\Sha^1_{\infty}(
\Q,R^1_{L/\Q}\Gm)=0$ and $\Sha^1_{\infty}(
\Q,T)=0$.\par
Since $T$ is connected, we can apply \cite[Proposition~6.10]{Sansuc1981} to the torsor $G\rightarrow X$ under $T$ and we get $\br{X}=\Pic T=\br_1 X$. Therefore, by Theorem \ref{main3}, it will be necessary and sufficient to show that $\delta_\R(X(\R))$ is not a subgroup in $\Ho^1(\R,T)$. Since $T(\R)$ is divisible and since $\Ho^1(\R,T)$ is $2$-torsion, the natural map $\Ho^1(\R,T[2])\to \Ho^1(\R,T)$ is an isomorphism. So $\Ho^1(\R,T)$ consists of the cocyles: $\{\xi\mapsto (\diag (a_1,...,a_{n-1}, \prod_i a_i^{-1})\,|\,a_i\in \{\pm 1\}\}$. The image of $\delta_\R(X(\R))$ consists of all cocycles $\xi\to m$, such that $J_{p,q}\cdot m$ has signature $(p,q)$. Assume without loss of generality that $p\geq 2$ (else we have $q\geq 2$ and this case is analogous). The cocycles $\xi\mapsto \diag(-1,1,...,-1)$ and $\xi\mapsto \diag(1,-1,...,-1)$ (where dots mean that all entries there are $1$'s) both lie in $\delta_\R(X(\R))$, however their product does not.\par
So we conclude that we have $\overline{X(\Q)}^{\infty}\subsetneq X(\A_\Q^{\infty})^{\br'_{\infty}}$ in this case.
\end{proof}
\end{prop}

\par Before giving another example with finite commutative stabilizer, we need the following lemma on cup products in group cohomology. 

\begin{lem}\label{cupcomp}
Let $G$ and $H$ be groups that act trivially on a commutative ring $R$. Denote the projections of $G\times H$ on the first and second factor by $p_1$ and $p_2$ respectively. The following statements hold:
\begin{itemize}
    \item Given a section $s$ to $p_2$, written $s=(\varphi_s,\Id):H\to G\times H$, we have a commutative diagram:
\begin{equation*}
\begin{tikzcd}
{\Ho^p(G,R)\otimes \Ho^q(H,R)} \arrow[d, "p_1^*\cup p_2^*"] \arrow[r, "\varphi_s^*\otimes \Id"] & {\Ho^p(H,R)\otimes \Ho^q(H,R)} \arrow[d, "\cup"] \\
{\Ho^{p+q}(G\times H,R)} \arrow[r, "\varphi_s^*"]                                                 & {\Ho^{p+q}(H,R)};                                
\end{tikzcd}
\end{equation*}
    \item Given homomorphisms $f_1\colon G'\to G$ and $f_2\colon H'\to H$, there is a commutative diagram:
  \begin{equation*}
\begin{tikzcd}
{\Ho^p(G',R)\otimes \Ho^q(H',R)} \arrow[d, "p_1'^*\cup p_2'^*"] \arrow[r, "f_1^*\otimes f_2^*"] & {\Ho^{p}(G,R)\otimes \Ho^q(H,R)} \arrow[d, "p_1^*\cup p_2^*"] \\
{\Ho^{p+q}(G'\times H',R)} \arrow[r, "(f_1\times f_2)^*"]                                       & {\Ho^{p+q}(G\times H,R)}.               \end{tikzcd}
\end{equation*}
\end{itemize}
\end{lem}
\begin{proof}
We will prove the statements by doing calculations with cocycles, since it suffices to show that the diagram commutes on the level of cocycles.
Denote by $[-]$ the cohomology class of a cocycle.
To prove that the first diagram commutes, take cocycles $\alpha\colon G^p\to R$ and $\beta\colon H^q\to R$.
We have $\varphi_s^*(\alpha)=\alpha\circ (\varphi_s)^p\colon H^p\to R$.
Then we get that $[\alpha\circ \varphi_s]\cup [\beta]$ is represented by the cocycle $$H^{p+q}\to R\quad (h_1,...,h_p,...,h_{p+q})\mapsto \alpha(\varphi_s(h_1),...,\varphi_s(h_p))\cdot \beta(h_{p+1},...,h_{p+q}).$$ On the other hand, applying $p_1^*\cup p_2^*$ to $[\alpha]\otimes [\beta]$ gives a cohomology class that is represented by the cocycle\\ $$(G\times H)^{p+q}\to R\quad ((g_1,h_1),...,(g_{p},h_p),...,(g_{p+q},h_{p+q}))\mapsto \alpha(g_1,...,g_p)\cdot \beta(h_{p+q},...,h_{p+q}).$$ Pulling this cocycle back via $s^{p+q}\colon H^p\to G^{p+q}\times H^{p+q}$ gives us that the first diagram commutes.\\
One proves that the second diagram commutes in an analogous way by computing that for $\alpha\otimes \beta$ a tensor product of two cocycles, one ends up with the cocycle: $$(G\times H)^{p+q}\to R\quad ((g_1,h_1),...,(g_p,h_p),...,(g_{p+q},h_{p+q}))\mapsto \alpha(f_1(g_1),...,f_1(g_p))\cdot \beta(f_2(g_{p+1}),...,f_2(g_{p+q}))$$
\end{proof}

Now we give an example in which the stabilizer is finite and in which we have the strict inclusion
$\overline{X(\Q)}^{\infty}\subsetneq X(\A_\Q^\infty)^{\br'_{\infty}}$. Note the contrast between this and the example in Proposition~\ref{egSU/T2}.

\begin{prop}\label{egSU/T4}
Set $\mu=T[4]$ to be the subgroup of $\SU_{2,1}$ as defined in Definition \ref{diagonalT}. Define the $\Q$-variety $X:=\SU_{2,1}/\mu$. We have the strict inclusion:
\begin{equation*}\overline{X(\Q)}^{\infty}\subsetneq X(\A_\Q^\infty)^{\br'_{\infty}}.\end{equation*}
\end{prop}

To prove this result, we will reduce the problem step by step via the following three lemmas. Throughout these lemmas we keep the notation as in Proposition \ref{egSU/T4}. Let $v$ be a place of~$\Q$. Elements in $\br X$ can be base-changed to elements in $\br X_{v}$, after which they can be paired with elements in $\Ho^1(\Q_v,\mu). $ This induces a pairing $\Ho^1(\Q_v,\mu)\times \br X\to \Q/\Z$ which is compatible with the Brauer-Manin pairing (see Proposition \ref{extended pairing} and Remark \ref{descibedpairing}).

\begin{lem}\label{reducttolinBrXR}
If the pairing $\Ho^1(\R,\mu)\times \br(X_\R)\to \frac{1}{2}\Z/\Z$ is linear on the left, then Proposition~\ref{egSU/T4} holds.
\begin{proof}
The elements in $\Ho^1(\R,\mu)$ correspond to elements in $T[2]$ because $\mu$ is constant over~$\R$. Under this identification, $\Ho^1(\R,\mu)$ consists of elements of the form $\diag(a_1,a_2,\frac{1}{a_1a_2})$ where we have $a_i\in \{\pm 1\}$. By Lemma \ref{realhermitianform}, the image of such an element $\diag(a_1,a_2,\frac{1}{a_1a_2})$ under the connecting homomorphism $\delta:\Ho^1(\R,\mu)\to \Ho^1(\R,\SU_{2,1})$ is represented by the hermitian form $\diag(a_1,a_2,\frac{-1}{a_1a_2})$. So we see that the element $\diag(-1,-1,1)$ maps to a form with signature $(-1,-1,-1)$ while all the other three elements in $\Ho^1(\R,\mu)$ map to forms with signature $(1,1,-1)$, i.e. precisely one element in $\Ho^1(\R,\mu)$ maps to a nonzero element in $\Ho^1(\R,\SU_{2,1})$.\par
We let $y$ be an element in $\Ho^1(\Q,\mu)$ whose image in $\Ho^1(\R,\SU_{2,1})$ under the canonical map is nonzero. This $y$ exists by the previous paragraph combined with $\Ho^1(\Q,\mu)\to \Ho^1(\R,\mu)$ being surjective (see Lemma \ref{surjective}). Consider the diagonal image of $y$ in ${\prod}^\prime_{v\neq \infty}\Ho^1(\Q_v,\mu)$, which lifts to an adelic point $(x_v)_{v\neq \infty}$ by Lemma \ref{kneserharder} combined with Proposition \ref{pointedsets}. For any $\beta\in \br{X}$, we have $\langle  (x_v)_{v\neq\infty}, \beta\rangle=\sum_{v\neq \infty} \langle y|_v,\beta|_v\rangle=\langle y|_{\infty},\beta|_{\infty}\rangle$, where the last equality comes from the fact that we have $\sum_{v\in \Omega_\Q}\langle y|_v,\beta_v\rangle=0$. 
Write $y|_\infty=y_1+y_2$, where $y_1,y_2\in \Ho^1(\R,\mu)$ are both different from $y|_{\infty}$. The linearity assumption in the lemma implies that we have $$\langle y|_{\infty},\beta|_{\infty}\rangle =\langle y_1,\beta|_{\infty}\rangle+\langle y_2,\beta|_{\infty}\rangle.$$ If we assume that $\beta$ is inside $\br'_{\infty} X$, then this sum vanishes, because both $y_1$ and $y_2$ come from points in $X(\R)$ by the previous paragraph. Therefore we have $(x_v)_{v\neq \infty}\in X(\A_\Q^{\infty})^{\br'_{\infty}}$.\par
Since $\mu=T[4]$ is constant over $\Q$ by Lemma \ref{diagonalTproperties}, we have by \cite[Lemme 1.1]{Sansuc1981} the equality $\Sha^1_{\infty}(\Q,\mu)=0$. Hence $y$ is the unique element in $\Ho^1(\Q,\mu)$ that restricts to $(y|_v)_{v\neq \infty}$. Combining this with the fact that the image of $y$ is nonzero in $\Ho^1(\R,\SU_{2,1})$ yields that we have $(x_v)_{v\neq\infty}\notin\overline{X(\Q)}^\infty$ by Lemma~\ref{S-closure}, which proves the lemma.
\end{proof}
\end{lem}

Recall from Remark \ref{descibedpairing} that for any field extension $K$ of $k$, there is a canonical isomorphism $\Ho^2(\pi_1(X_K),\overline{K}^\times)\cong \br X_K$, which is compatible with the pairing with $\Ho^1(K,\mu)$. So in Lemma~\ref{reducttolinBrXR}, we could have also written $\Ho^2(\pi_1(X_\R),\C^\times)$ instead of $\br X_\R$. The following lemma aims to replace this term by $\Ho^2(\pi_1(X_\R),\Z/2)$. Recall that elements of $\Ho^1(K,\mu)$ may be thought of as equivalence classes of sections $s:\Gamma_K\to \pi_1(X_K)$.

\begin{lem}\label{suffleftlinear}
If the pairing  $\Ho^1(\R,\mu)\times \Ho^2(\pi_1(X_\R),\Z/2)\to \Ho^2(\R,\Z/2),\quad  (s,\alpha) \mapsto s^*\alpha$ is linear on the left, then the pairing $\Ho^1(\R,\mu)\times \br X_\R\to \Q/\Z$ is linear on the left. 
\begin{proof}
We identify $\Z/2$ with $\mu_2\subset \C^\times$. By using the exact sequence $0\to \Z/2\to \C^\times \to \C^\times\to 0$ of $\pi_1(X_\R)$-modules, we obtain a surjection $\Ho^2(\pi_1(X_\R),\Z/2)\to \Ho^2(\pi_1(X_\R),\C^\times)[2]$. Take an element in $\Ho^1(\R,\mu)$, which is represented by a section $s:\Gamma_\R\to \pi_1(X_\R)$. There is a commutative diagram
\begin{equation*}
\begin{tikzcd}
{\Ho^2(\pi_1(X_\R),\Z/2)} \arrow[r] \arrow[d, "s^*"] & {\Ho^2(\pi_1(X_\R),\C^{\times})} \arrow[d, "s^*"] \\
{\Ho^2(\R,\Z/2)} \arrow[r, "\sim"]                               & \br{\R}.                                         
\end{tikzcd}
\end{equation*}
Hence the map $\Ho^2(\pi_1(X_\R),\Z/2)\to \Ho^2(\pi_1(X_\R),\C^{\times})$ is compatible with the pairing with $\Ho^1(\R,\mu)$. Since $\pi_1(X_\R)$ is $4$-torsion, we have a surjection $\Ho^2(\pi_1(X_\R),\mu_4)\to\Ho^2(\pi_1(X_\R),\C^\times)$. By using Magma (the code is included in Appendix \ref{magma}), we verify that we have $\Ho^2(\pi_1(X_\R),\mu_4)\cong (\Z/2)^6$ and thus $\Ho^2(\pi_1(X_\R),\C^\times)$ is $2$-torsion. Therefore the map $\Ho^2(\pi_1(X_\R),\Z/2)\to \Ho^2(\pi_1(X_\R),\C^{\times})$ is surjective, which proves the lemma.
\end{proof}
\end{lem}

We now prove the remaining step in the proof of Proposition \ref{egSU/T4}, which is proving the assumption in Lemma \ref{suffleftlinear}. We will decompose $\Ho^2(\pi_1(X_\R),\Z/2)$ into a direct sum of submodules by using the Künneth formula and we will study the pairing with $\Ho^1(\R,\mu)$ on these submodules.

\begin{lem}\label{pairingisleftlinear}
The pairing $\Ho^1(\R,\mu)\times \Ho^2(\pi_1(X_\R),\Z/2)\to \Ho^2(\R,\Z/2)$ is linear on the left.
\end{lem}
\begin{proof}
By Lemma \ref{diagonalTproperties}, we have that $\mu$ is constant over $\R$, so we get $\pi_1(X_\R)\cong \mu(\C)\times \Gamma_\R$. Then we can apply the Künneth formula to decompose $\Ho^2(\pi_1(X_\R),\Z/2)$ as follows (note that we omitted tensoring with $\Ho^0(\mu(\C),\Z/2)$ and $\Ho^0(\R,\Z/2)$ in the formula to improve the presentation, but omitting this tensoring makes no difference):
\begin{equation*}
p_1^*\cup p_2^*\colon \Ho^2(\R,\Z/2)\oplus [\Ho^1(\mu(\C),\Z/2)\otimes \Ho^1(\R,\Z/2)]\oplus \Ho^2(\mu(\C),\Z/2)\overset{\sim}{\to}\Ho^2(\pi_1(X_\R),\Z/2).
\end{equation*}
Pairing elements in $\Ho^2(\R,\Z/2)=\Ho^2(\R,\Z/2)\otimes \Ho^0(\mu(\C),\Z/2)$ with elements in $\Ho^1(\R,\mu)$ is bilinear because by the first part of Lemma \ref{cupcomp} we have that for any $s\in\Ho^1(\R,\mu)$, the restriction of $s^*$ to $\Ho^2(\R,\Z/2)$ is the identity.\par
It also follows from Lemma \ref{cupcomp} that pairing with elements in $\Ho^1(\mu(\C),\Z/2)\otimes \Ho^1(\R,\Z/2)$ is linear since for $f\otimes g$ in this group and for sections $s,s':\Gamma_\R\to \mu(\C)\times \Gamma_\R$ to $p_2$, we have $$(s+s')^*(p_1^*(f)\cup p_2^*(g))=(p_1\circ s)^*(f)\cup g+(p_1\circ s')^*(f)\cup g.$$ 
Then the remaining part is to show that 
\begin{equation}\label{pairingwithmu}\Ho^2(\mu(\C),\Z/2)\times \Ho^1(\R,\mu(\C))\to \br{\R}\end{equation}
is linear on the left. Since $\mu_\R$ is constant, we have $\Ho^1(\R,\mu)=\Hom(\Gamma_\R,\mu(\C))$ and by Lemma~\ref{cupcomp}, the pairing (\ref{pairingwithmu}) is just the pullback $(f,\alpha)\mapsto f^*(\alpha)\in \Ho^2(\R,\Z/2)\cong \br \R$. \par The group $\Ho^2(\mu(\C),\Z/2)$ classifies central extensions of $\mu(\C)$ by $\Z/2$. Picking a basis gives $\mu(\C)\cong (\Z/4)^2$ and we are taking coefficients in the field $\Z/2$, so the Künneth-formula applies: \begin{equation*}p_1^*\cup p_2^*\colon \bigoplus_{p+q=2}\Ho^p\otimes \Ho^q\coloneq \bigoplus_{p+q=2} \Ho^p(\Z/4,\Z/2)\otimes \Ho^q(\Z/4,\Z/2)\overset{\sim}{\to} \Ho^2(\mu(\C),\Z/2).\end{equation*}
Since $\Z/4$ has one nontrivial central $\Z/2$-extension, $[\Z/8]$, which is abelian, the terms $\Ho^2\otimes \Ho^0$ and $\Ho^0\otimes \Ho^2$ are isomorphic to $\Z/2$ and have as nontrivial element an abelian extension. Hence any nonabelian extensions of $\mu(\C)$ by $\Z/2$ has $\Ho^1\otimes \Ho^1$-component not equal to zero.
It is well known that the pairing restricted to abelian extensions \[\Hom(\Gamma_\R,(\Z/4)^2)\times \Ext_{\mathrm{ab}}((\Z/4)^2,\Z/2)\to \Ext_{\mathrm{ab}}(\Gamma_\R,\Z/2)=\Ho^2(\R,\Z/2)\] is linear on the left, because for any abelian group $B$, the functor $\Ext_{\mathrm{ab}}(-,B)$ is additive (because $\Ext_{\mathrm{ab}}(-,B)$ is the first derived functor of $\Hom(-,B)$ and the $\Hom$-functor is additive. Note that bi-additivity of $\Ext_{\mathcal{A}}(-,-)$ holds in any abelian category $\mathcal{A}$ by a tedious calculation with the Bear-sum). This takes care of the $\Ho^0\otimes \Ho^2$ and $\Ho^2\otimes \Ho^0$ component of $\Ho^2(\mu(\C),\Z/2)$.
Therefore it suffices to show that any $\alpha\in \Ho^1\otimes \Ho^1\subset \Ho^2(\mu(\C),\Z/2)$ pairs trivially with $\Hom(\Gamma_\R,(\Z/4)^2)$.
Pick any element $\varphi=(\varphi_1,\varphi_2)\in \Hom(\Gamma_\R,(\Z/4)^2)$ and consider the commutative diagram (see Lemma~\ref{cupcomp}):
\begin{equation*}
\begin{tikzcd}
{\Ho^1(\Z/4,\Z/2)\otimes \Ho^1(\Z/4,\Z/2)} \arrow[d, "p_1^*\cup p_2^*"] \arrow[rr, "\varphi_1^*\otimes \varphi_2^*"] &  & {\Ho^1(\R,\Z/2)\otimes \Ho^1(\R,\Z/2)} \arrow[d, "\cup"] \\
{\Ho^2((\Z/4)^2,\Z/2)} \arrow[rr, "\varphi^*"]                                                             &  & {\Ho^2(\R,\Z/2)}.                                                          
\end{tikzcd}
\end{equation*}
For both $i$ we have that $\varphi_i^*\colon \Ho^1(\Z/4,\Z/2)\to \Ho^1(\R,\Z/2)$ is the zero map, since $\varphi_i$ factors via $2\Z/4\to \Z/4$. Therefore the pairing with $\Ho^1\otimes \Ho^1$ is indeed zero, which proves that the pairing (\ref{pairingwithmu}) is linear on the left. As mentioned just above (\ref{pairingwithmu}), this is sufficient to prove the lemma.
\end{proof}

\begin{proof}[Proof of Proposition \ref{egSU/T4}] Combine Lemma \ref{reducttolinBrXR}, Lemma \ref{suffleftlinear} and Lemma \ref{pairingisleftlinear}.
\end{proof}

\section{An attempt to generalize the ambient group}
When considering more general homogeneous spaces $X$ of $G$ where $G$ is a connected linear $k$-group, a classical construction (as in \cite{borovoi1996brauer}, \cite{borovoi2008elementary} and \cite{borovoi2013manin}) is to consider an auxiliary homogeneous space $W$ under $G'$ such that the maximal quotient of multiplicative type of the stabilizer injects into the maximal torus quotient of the ambient group $G_1$, and there is a morphism $W\rightarrow X$ making $W$ into an $X$-torsor of a quasi-trivial torus. This construction relates the approximation properties proven for $W$ to those we wish for $X$. We hope to use the same construction to study the question of when $\overline{X(k)}^S$ can be precisely cut out by $\br_S X$ or $\br_S^\prime X$ inside $X(\A_k^S)$, for homogeneous spaces $X=G/H$ of a connected linear $k$-group $G$ with commutative stabilizer $H$. We show that positive results still hold for the auxiliary $W$, but new difficulties arise when relating $W$ to $X$.
\par We recall the following notation as in \cite{borovoi1996brauer}: $G^\circ$ is the neutral connected component of $G$; $\Gu$ is the unipotent radical of $G^\circ$; $\Gred\coloneq G^\circ/\Gu$ (it is a connected reductive algebraic group); $\Gss$ is the derived group of $\Gred$ (it is semisimple); $\Gtor\coloneq\Gred/\Gss$ (it is a torus); $\Gssu\coloneq\ker(G^\circ\rightarrow\Gtor)$ (it is an extension of $\Gss$ by $\Gu$); $\Gmult\coloneq G/\Gssu$ (it is of multiplicative type). We denote by $X(\A_k)_\bullet$ the modified adelic space such that each $X(k_v)$ for $v$ archimedean is collapsed to its connected components, and the usual Brauer-Manin pairing induces a pairing with $X(\A_k)_\bullet$ since an evaluation map is constant on a connected component.
\begin{prop}\label{generalisation}
Let $W=G_1/H_1$ be a homogeneous space of a connected linear $k$-group $G_1$ with commutative stabilizer $H_1$, such that $\Gssu_1$ is simply connected and satisfies strong approximation away from $S$. Suppose $\Hmult_1$ injects into $\Gtor_1$. Then we have the equality\[\overline{W(k)}^S=W(\A_k^S)_\bullet^{\br_{1,S}}.\]

\end{prop}
\begin{proof}Set $(x_v)_{v\notin S}\in W(\A_k^S)_\bullet^{\br_{1,S}}$ and consider any open $U$ containing $(x_v)_{v\notin S}$. We want to find $x\in W(k)$ such that $x\in U$. As in the proof of \cite[Proposition 3.5]{borovoi1996brauer}, we set $Y=\Gssu_1\backslash W$ and we let $\varphi\colon  W\rightarrow Y$ be the canonical map. Then $Y$ is a $k$-torus; the fibers of $\varphi$ are homogeneous spaces of the simply connected group $\Gssu_1$ with stabilizer 
isomorphic to $H_1\cap \Gssu_1$. 
 Since $\Hmult_1$ injects into $\Gtor_1$, we have $H_1\cap \Gssu_1\simeq H_1^{\text u}$, which is isomorphic to $\Ga^r$ for some $r$. For any $k$-point $z\in Y(k)$, the fiber $\varphi^{-1}(z)$ also has a $k$-point (by the vanishing of $\Ho^1(k,\G_a^r)$). We denote by $(y_v)_{v\notin S}$ the point $\varphi((x_v)_{v\notin S})\in Y(\A_k^S)$. Then we have $(y_v)_{v\notin S}\in Y(\A_k^S)_\bullet^{\br_{1,S}}$ by the commuting diagram 
\begin{equation}
\begin{tikzcd}
W(\A_k^S)_\bullet \arrow[d] \arrow[r,phantom,"\times" description] & \br_{1,S} W \arrow[r] \arrow[d,leftarrow]      & \QZ \arrow[d,Rightarrow, no head] \\
Y(\A_k^S)_\bullet                                \arrow[r,phantom,"\times" description] & {\br_{1,S} Y} \arrow[r]  & \QZ.     
\end{tikzcd}
\end{equation}
The map $\varphi$ induces an open map $\varphi\colon  W(\A_k^S)_\bullet\rightarrow Y(\A_k^S)_\bullet$ (see the proof of \cite[Lemma 4.2]{borovoi2013manin}), and thus we get an open $\varphi(U)$ containing $(y_v)_{v\notin S}.$
Since $Y$ is a $k$-torus, we have that $\overline{Y(k)}^S=Y(\A_k^S)^{\br_{1,S}}_\bullet$ (see \cite[Theorem 1]{harari2010brauer} or \cite[Proposition 4.6(1)]{liu2015very}, whose proofs are in fact valid for an arbitrary torus $T$ and an arbitrary set $S$ of places). Hence we can find $y\in Y(k)\cap\varphi(U).
$ Let $W_y$ be the fiber of $W$ over $y$, and set $V=W_y(\A_k^S)_\bullet\cap U\subseteq W(\A_k^S)_\bullet.$ Then $W_y=\Gssu_1/\Ga^r$ satisfies strong approximation away from $S$ by the assumption for $\Gssu_1$ and the vanishing of $\Ho^1(k_v,\Ga^r).$ Hence we can find $x\in W_y(k)\cap V\subseteq W(k)\cap U$ which finishes the proof.
\end{proof}
Set $X=G/H$ with $G$ being connected linear and $H$ being commutative. Is it possible that Theorem~\ref{main3} remains true for $X$ with the assumption of $\Gssu$ being simply connected and satisfying strong approximation away from $S$? Let us first construct the auxiliary $W$, as in \cite[\S 4.3]{borovoi1996brauer}. Choose an embedding $j : \Hmult \hookrightarrow T$ of $\Hmult$ into a quasi-trivial
torus $T$. Then we define an embedding $H\rightarrow G\times T, h
\mapsto (i(h),j(\mu(h)))$ where $i: H\hookrightarrow G$ and $\mu: H\rightarrow \Hmult$ are the canonical morphisms. Set $W=(G\times T)/H$. Then $W\rightarrow X$ is a torsor under $T$ and $W$ satisfies the condition that  $\Hmult$ injects into $(G\times T)^{\mathrm{tor}}$. \par
Supposing $\Gssu$ is simply connected and satisfies strong approximation away from $S$, we have $\overline{W(k)}^S=W(\A_k^S)_\bullet^{\br_{1,S}}.$ If we can lift an $S$-adelic point $(x_v)_{v\notin S}\in X(\A_k^S)_\bullet^{\br_{1,S}}$ (or in $X(\A_k^S)_\bullet^{\br^
\prime_{1,S}}$ or $X(\A_k^S)_\bullet^{\br^
\prime_{S}}$) to an $S$-adelic point in $W(\A_k^S)_\bullet^{\br_{1,S}}$, then we will get that $(x_v)_{v\notin S}\in \overline{X(k)}^S$. Our examples of $\overline{X(k)}^S\subsetneq X(\A_k^S)^{\br{1,S}}$ (or $\overline{X(k)}^S\subsetneq X(\A_k^S)^{\br^\prime_{1,S}}$  or $\overline{X(k)}^S\subsetneq X(\A_k^S)^{\br^\prime_{S}}$) in the previous section show that we need specific conditions for such a lifting to happen, unlike the situation in \cite[Corollary 7.21]{borovoi2013manin} where the lifting happens unconditionally. \par We also note that a $W$ as in Proposition \ref{generalisation} does not necessarily satisfy the condition (\ref{brauer1intro}) in Theorem \ref{main3}: take one of the many $X=G/H$ constructed in the previous section such that (\ref{brauer1intro}) is not satisfied. Then for an auxiliary space $W=(G
\times T)/H$ we have that $W(k_v)\rightarrow \Ho^1(k_v,H)$ factors through $W(k_v)\rightarrow\ X(k_v)\rightarrow  \Ho^1(k_v,H)$, violating (\ref{brauer1intro}) too. Therefore, we cannot naively restate Theorem~\ref{main3} for a homogeneous space $X=G/H$ of a connected linear $G$ with commutative stabilizer $H$, with the assumption of $\Gssu$ being simply connected and satisfying strong approximation away from $S$; we might need to state conditions on $\Gssu/\Gssu\cap H$ relating to the cases we studied.

\section{Appendix}
\subsection{ The real cohomology of $\SU_{p,q}$}
Let $\SU_{p,q}$ and $J_{p,q}$ be as in Definition \ref{SUpq}. In this appendix we aim to describe the real Galois cohomology of $\SU_{p,q}$. Throughout we fix $p,q$ and $n\coloneq p+q$. We will give a description of the following bijection in both directions:  
\[\Ho^1(\R,\SU_{p,q})\leftrightarrow \{\text{nonsingular hermitian forms }\C^n\times \C^n\to \C\text{ with discriminant }(-1)^p\}.\] 
 See \cite[29.19]{bookofinvolutions} and  \cite{allenhermitianforms} for results in a more general setting. 
\par We start by giving the definition of a hermitian form and some properties that a hermitian form can have.

\begin{df}
Let $k$ be a field and let $A$ be a quadratic étale algebra over $k$ with $\xi\colon a\mapsto \overline{a}$ the nontrivial element in $\Aut(A/k)$. A hermitian form over $k$ (with respect to $A/k$) is a $k$-bilinear map $h\colon A^n\times A^n\to A$  that satisfies $h(a_1,a_2)=\overline{h(a_2,a_1)}$ for any $a_1,a_2\in A^n$.
\end{df}

By picking an $A$-basis of $A^n$, we get that $h\colon A^n\times A^n\to A$ is given by $(v,w)\mapsto \overline{v}^tM w$ for some $M\in \mathrm{M}_n(A)$ such that $M_{ij}=\overline{M_{ji}}$ for all entries. The \textit{discriminant} of a hermitian form is $[\det(M)]\in k^*/(k^*)^2$.\par
We say that two hermitian forms $h,h'\colon A^n\times A^n\to A$ are \textit{isomorphic} over $k$ if there exists some $g\in \GL_n(A)$ such that the following diagram commutes: 
\begin{equation*}
\begin{tikzcd}
A^n\times A^n \arrow[r, "h"] \arrow[d, "g\times g"] & \C. \\
A^n\times A^n \arrow[ru, "h'"']                     &   
\end{tikzcd}  
\end{equation*}
In terms of corresponding matrices $M$ and $M'$, this is equivalent to there being some $S\in \GL_n(A)$ such that $SM\overline{S}^t=M'$ and we then call the matrices $M,M'$ \textit{similar} over $k$.\par 
A hermitian form is called \textit{nonsingular} if the induced pairing is nondegenerate.\\

By Sylvester's law of inertia, any nonsingular hermitian form $h\colon \C^n\times \C^n\to \C$ over $\R$ can have its matrix put in the form $M=\overline{S}^t  J_{p',q'} S$ for some $p'+q'=n$. It is easy to check that distinct $J_{p',q'}$ matrices are not similar over $\R$ and hence the $\R$-isomorphism classes of hermitian forms are given precisely by $\{[J_{p',q'}]\,|\,p'+q'=n\}$. After applying $(-)\otimes_\R\C$, one gets a hermitian form $h_\C\colon \C^n\otimes_\R \C\to \C^n\otimes_\R \C\to \C\otimes_\R \C$ over $\C$ (with respect to the étale algebra $\C\otimes_\R \C/\C$). Denote by $S\mapsto S^*$ the involution on $\Mat_n(\C)\otimes_\R \C$ given by $A\otimes B\mapsto \overline{A}^t\otimes B$. One may apply $\Lambda^*(-)\Lambda$ to $J_{p',q'}$ for $\Lambda$ a diagonal matrix with coefficients either $1\otimes 1$ or $1\otimes i$ to get $J_{p',q'}$ similar to $J_{p,q}$. So any hermitian form over $\R$ becomes isomorphic over $\C$ to the hermitian form defined by $J_{p,q}$.\\

The complex conjugation $\xi$ also acts on $\Mat_n(\C)\otimes_\R \C$ by conjugating the second factor. We associate to $h$ with matrix $M_h$ a $1$-cocycle $\sigma_h\colon \Gamma_\R\to \SU_{p,q}(\C)$ by first picking an isomorphism $S\in \GL(\C^n\otimes_\R \C)$,  $(\C^n\otimes_\R \C, h)\overset{S}{\to}(\C^n\otimes_\R \C, J_{p,q})$ and then setting $\sigma_h(\xi)=S^{-1}\xi(S)$ (note that $\xi$ acts only on the second factor of $\C\otimes_\R \C$). Since $(S^*)^{-1} J_{p,q} (S^{-1})^*=M_h$ we get that $\det(S)\in \R^\times$ and thus $\det(S^{-1}\xi(S))=1$. Moreover, to show that the cocycle maps into $\SU_{p,q}(\C)$, we compute:
\begin{equation*}
(S^{-1}\cdot \xi(S))^*\cdot (J_{p,q}\otimes 1)\cdot S^{-1}\xi(S)= \xi (S)^*\cdot (M_h\otimes 1)\cdot \cdot \xi(S)=\xi(S^*\cdot (M_h\otimes 1)\cdot S)=J_{p,q}\otimes 1.
\end{equation*}
It is routine to show that the cohomology class of $\sigma_h$ does not depend on the choice of $S$ and the choice of $h$ in an isomorphism class of hermitian forms.\\

To give the inverse: Suppose that $\sigma\colon \Gamma_\R\to \SU_{p,q}(\C)$, defined by $\xi\mapsto M$, is a cocycle. Consider the natural inclusion $\SU_{p,q}(\C)\subseteq \SL_{2n}(\C)$. Since $\Ho^1(K,\SL_{m})=0$ for any field $K$ and any $m$, there is $T\in \SL_{2n}(\C)$ such that $T^{-1}\xi(T)=M$. Then consider the bilinear form on $\C^{n}\otimes_\R\C$, restricted to~$\C^n$: \begin{equation*}b_\sigma\colon \C^n\times\C^n\to \C\otimes_\R \C\quad (v,w)\mapsto (v\otimes 1)^t\cdot  T^*\cdot (J_{p,q}\otimes 1)\cdot T\cdot (w\otimes 1).\end{equation*}
Notice that since $T\mapsto T^*$ commutes with complex conjugation on $\Mat_n(\C\otimes_\R \C)$, we have: \begin{equation*}\xi(T^*\cdot (J_{p,q}\otimes 1)\cdot T)=\xi(T)^*\cdot (J_{p,q}\otimes 1)\cdot \xi(T)= (M\cdot T)^*\cdot (J_{p,q}\otimes 1)\cdot (M\cdot T)=T^*\cdot (J_{p,q}\otimes 1)\cdot T.\end{equation*}
So actually the matrix defining $b_\sigma$ lies in $\Mat_n(\C)\otimes 1\subseteq \Mat_n(\C\otimes_\R \C)$. It follows directly that $T^*\cdot (J_{p,q}\otimes 1)\cdot T$ is a hermitian matrix and hence it defines a hermitian form $h_\sigma\colon \C^n\times \C^n\to \C$. This hermitian form becomes isomorphic to $J_{p,q}$ over $\C$ by the isomorphism $T$, which by construction gives the cocycle $\sigma(\xi)=T^{-1}\xi(T)=M$.\\

It is clear that the assingments described above are each others inverses and hence we get the bijection:
\begin{equation*}
\Ho^1(\R,\SU_{p,q})\leftrightarrow \{\text{nonsingular hermitian forms }\C^n\times \C^n\to \C\text{ with discriminant }(-1)^p\}.
\end{equation*}

\begin{eg}
Let $\sigma\colon \Gamma_\R\to \SU_{p+q}(\C)$ be a cocycle defined by $\sigma(\xi)=\colon M\in \SU_{p+q}(\R)$. Then $M$ is $2$-torsion and hence diagonalizable over $\R$, so $M=g^{-1}\cdot J_{p',q'}\cdot g$ for $q'\equiv q\,(\text{mod }2)$. Therefore the cocycle $\sigma$ is equivalent to $\tau\colon\xi\mapsto J_{p',q'}$. We write $\C^n\otimes_\R\C$ as $\C^n\oplus j\cdot \C$ with $j^2=-1$. Consider the tensor product of diagonal matrices, where there are $q'$ $i$'s and $q'$ $j$'s. Let 
\begin{equation*}
T\coloneq \diag(i,...,i,1,...,1)\otimes \diag(j,...,j,1,...,1).
\end{equation*}
It follows that $T^{-1}\xi(T)=J_{p',q'}$ and thus the associated hermitian form $h_\sigma$ is then defined by the matrix $T \cdot (J_{p,q}\otimes 1)\cdot T^*=(J_{p,q}\otimes 1)\cdot (\diag(1,...,1,j,...,j))^2=J_{p,q}\cdot J_{p',q'}$. So if $\xi$ has real image $M$ under $\sigma$, the corresponding hermitian form is obtained by multiplying $J_{p,q}$ by $J_{p',q'}$ where $q'$ is the number of negative eigenvalues of $M$.
\end{eg}

\subsection{Magma code}\label{magma}
We include here the Magma code that we used to calculate $\Ho^2(\pi_1(X_\R),\mu_4)$ in the proof of Lemma \ref{suffleftlinear}:
\begin{verbatim}
Pi1:= AbelianGroup(GrpPC, [2,4,4]);
Pi1:= ConditionedGroup(Pi1);
CM:= CohomologyModule(Pi1, [4], [3,1,1,1,1]);
H2:= CohomologyGroup(CM, 2);
H2;
\end{verbatim}

\bibliographystyle{alpha}
\bibliography{bibliographie}

\end{document}

%% file: pre.tex
\usepackage{appendix}        \usepackage[alphabetic,lite]{amsrefs}                                                      
\usepackage{ulem}
\usepackage{array}
\usepackage{authblk}
\usepackage{hyperref}
\usepackage{bm}                                                                
\usepackage{colonequals}                                                       
\usepackage{color}  
\usepackage{dynkin-diagrams}                                                   
\usepackage{fullpage}                                                          
\usepackage{indentfirst}                             
\usepackage{moreenum}                                                          
\usepackage{stmaryrd}                                                          
\usepackage{verbatim}                                                          
\usepackage[all,cmtip]{xy}                                                  

\usepackage{amssymb}                                                           
\usepackage{amsthm}                                             

\usepackage{graphicx}       \usepackage{tikz-cd}                                                
\usepackage{mathrsfs}                                                          
\usepackage{mathtools}                         
\usepackage{wrapfig}
\usepackage[all,cmtip]{xy}                                                     

\usepackage[OT2,T1]{fontenc}
\DeclareSymbolFont{cyrletters}{OT2}{wncyr}{m}{n}
\DeclareMathSymbol{\Be}{\mathalpha}{cyrletters}{"42}                          
\DeclareMathSymbol{\Che}{\mathalpha}{cyrletters}{"51}                          
\DeclareMathSymbol{\Sha}{\mathalpha}{cyrletters}{"58}                          

\makeatletter
\DeclareRobustCommand\widecheck[1]{{\mathpalette\@widecheck{#1}}}
\def\@widecheck#1#2{%
    \setbox\z@\hbox{\m@th$#1#2$}%
    \setbox\tw@\hbox{\m@th$#1%
       \widehat{%
          \vrule\@width\z@\@height\ht\z@
          \vrule\@height\z@\@width\wd\z@}$}%
    \dp\tw@-\ht\z@
    \@tempdima\ht\z@ \advance\@tempdima2\ht\tw@ \divide\@tempdima\thr@@
    \setbox\tw@\hbox{%
       \raise\@tempdima\hbox{\scalebox{1}[-1]{\lower\@tempdima\box
\tw@}}}%
    {\ooalign{\box\tw@ \cr \box\z@}}}
\makeatother

\theoremstyle{plain}      \newtheorem{thm}{Theorem}[section]                   
\theoremstyle{plain}                   
\theoremstyle{plain}      \newtheorem{lem}[thm]{Lemma}                         
\theoremstyle{plain}                             
\theoremstyle{plain}      \newtheorem{cor}[thm]{Corollary}                     
\theoremstyle{plain}                         
\theoremstyle{plain}      \newtheorem{prop}[thm]{Proposition}                  
\theoremstyle{plain}      \newtheorem{conjecture}[thm]{Conjecture}             
\theoremstyle{definition} \newtheorem{rmk}[thm]{Remark}                        
\theoremstyle{definition}                      
\theoremstyle{definition} \newtheorem{df}[thm]{Definition}                     
\theoremstyle{definition}                   
\theoremstyle{definition} \newtheorem{eg}[thm]{Example}                        
\theoremstyle{definition}                        
\theoremstyle{definition}                        
\theoremstyle{definition}                      
\theoremstyle{definition}                    
\theoremstyle{definition}                  
\theoremstyle{definition}                        
\theoremstyle{definition}                       
\theoremstyle{definition}                    
\theoremstyle{definition}                  
\theoremstyle{definition} \newtheorem{construction}[thm]{Construction}         
\theoremstyle{definition}              
\theoremstyle{definition}                  
\theoremstyle{definition} \newtheorem{prop-df}[thm]{Proposition-Definition}    
\theoremstyle{definition} \newtheorem{thm-df}[thm]{Theorem-Definition}
\theoremstyle{definition}
\newtheorem*{construction*}{Construction}                                      
\newtheorem*{conjecture*}{Conjecture}                                          
\newtheorem*{hypothesis*}{Hypothesis}                                          
\newtheorem*{convention*}{Convention}                                          
\newtheorem*{notation*}{Notation}                                              
\newtheorem*{prop*}{Proposition}                                               
\newtheorem*{summary*}{Summary}                                                
\newtheorem*{qt*}{Question}                                                    
\newtheorem*{rmk*}{Remark}                                                     
\newtheorem*{fact*}{Fact}                                                      
\newtheorem*{lizi*}{Example}                                                   
\newtheorem*{df*}{Definition}                                                  
\theoremstyle{plain}
\newtheorem*{thm*}{Theorem}                                                    

\newcommand{\bconst}{\begin{construction}}                                     
\newcommand{\econst}{\end{construction}}                                       
\newcommand{\benum}{\begin{enumerate}[label={{\upshape(\alph*)}}]}             
\newcommand{\benuma}{\begin{enumerate}[label={{\upshape(\arabic*)}}]}          
\newcommand{\benumr}{\begin{enumerate}[label={{\upshape(\roman*)}}]}           
\newcommand{\eenum}{\end{enumerate}}
\newcommand{\bconj}{\begin{conjecture}}
\newcommand{\econj}{\end{conjecture}}
\newcommand{\bconjnn}{\begin{conjecture*}}
\newcommand{\econjnn}{\end{conjecture*}}
\newcommand{\begs}{\begin{eg}\hfill\benuma}                                    
\newcommand{\eegs}{\eenum\end{eg}}                                             
\newcommand{\brmks}{\begin{rmk}\hfill\benuma}                                  
\newcommand{\ermks}{\eenum\end{rmk}}                                           
\newcommand{\bdfs}{\begin{df}\hfill\benuma}                                    
\newcommand{\edfs}{\eenum\end{df}}                                             
\newcommand{\bitem}{\begin{itemize}}                                           
\newcommand{\eitem}{\end{itemize}}                                             
\newcommand{\be}{\begin{equation}}                                             
\newcommand{\ee}{\end{equation}}                                               
\newcommand{\benn}{\begin{equation*}}                                          
\newcommand{\eenn}{\end{equation*}}                                            
\newcommand{\bqt}{\begin{qt*}\rm}                                              
\newcommand{\eqt}{\end{qt*}}                                                   
\newcommand{\bqtr}{\begin{qt*}\rm\coLR}                                        
\newcommand{\eqtr}{\end{qt*}}                                                  
\newcommand{\beac}{\begin{equation}\begin{array}{c}}                           
\newcommand{\eeac}{\end{array}\end{equation}}                                  
\newcommand{\beqn}{\begin{eqnarray*}}
\newcommand{\eeqn}{\end{eqnarray*}}
\newcommand{\bdf}{\begin{df}}
\newcommand{\bdfhf}{\begin{df}\hfill}
\newcommand{\edf}{\end{df}}
\newcommand{\brmk}{\begin{rmk}}
\newcommand{\brmkhf}{\begin{rmk}\hfill}
\newcommand{\ermk}{\end{rmk}}

\newcommand{\BA}{\mathbf{A}}

\newcommand{\BBG}{\mathbb{G}}

\newcommand{\A}{\mathbb{A}}                                                    
\newcommand{\C}{\mathbb{C}}                                                    
\newcommand{\G}{\mathbb{G}}                                                    
\newcommand{\Q}{\mathbb{Q}}                                                    
\newcommand{\R}{\mathbb{R}}                                                    
\newcommand{\Z}{\mathbb{Z}}                                                    
\newcommand{\QZ}{\mathbb{Q}/\mathbb{Z}}                                        
\newcommand{\coLR}{\textcolor[rgb]{1.00,0,0}}                                  

\DeclareMathOperator{\Hom}{Hom}                                                
\DeclareMathOperator{\Aut}{Aut}                                                
\DeclareMathOperator{\Ext}{\mathbf{Ext}}                                       

\DeclareMathOperator{\Mat}{Mat}                                                
\DeclareMathOperator{\Pic}{Pic}                                                
\DeclareMathOperator{\br}{Br}                                                  

\DeclareMathOperator{\Ker}{Ker}                                                
\renewcommand{\ker}{\Ker}                                                      
\DeclareMathOperator{\Coker}{Coker}		                                       

\DeclareMathOperator{\Id}{Id}                                                  
\DeclareMathOperator{\ev}{ev}                                                  
\DeclareMathOperator{\disc}{disc}                                              

\DeclareMathOperator{\Spec}{Spec}                                              

\newcommand{\et}{\mathrm{\acute{e}t}}                                          

\DeclareMathOperator{\inv}{inv}                                                

\DeclareMathOperator{\diag}{diag}                                              
\DeclareMathOperator{\pr}{pr}                                                  

\newcommand{\gm}{\BBG_m}                                                       
\DeclareMathOperator{\GL}{\mathbf{GL}}                                         
\DeclareMathOperator{\SL}{\mathbf{SL}}                                         
\DeclareMathOperator{\SU}{\mathbf{SU}}                                         
\DeclareMathOperator{\SO}{\mathbf{SO}}                                         
\DeclareMathOperator{\Spin}{\mathbf{Spin}}                                     
\newcommand{\Gred}{G^{\mathrm{red}}}                                           
\newcommand{\Gss}{G^{\mathrm{ss}}}    
\newcommand{\Gu}{G^{\mathrm{u}}}

\newcommand{\Gssu}{G^{\mathrm{ssu}}}                                           
\newcommand{\Gsc}{G^{\mathrm{sc}}}                                             
\newcommand{\Gmult}{G^{\mathrm{mult}}}                                         
\newcommand{\Hmult}{H^{\mathrm{mult}}}                                         
\newcommand{\Gtor}{G^{\mathrm{tor}}}                                           
\DeclareMathOperator{\res}{res}                                                
    
\newcommand\addtag{\refstepcounter{equation}\tag{\theequation}}

\newcommand{\QQ}{\mathbb{Q}}
\newcommand{\RR}{\mathbb{R}}
\newcommand{\ZZ}{\mathbb{Z}}

\newcommand{\Gm}{\mathbb{G}_m}
\newcommand{\Ga}{\mathbb{G}_a}